\documentclass[a4paper,11pt]{amsart} 
\usepackage{amsmath,amsxtra,amssymb,latexsym, amscd,amsthm}
\usepackage[mathscr]{eucal}
\usepackage{mathrsfs}
\usepackage{bbm}
\usepackage{enumerate}
\usepackage{pict2e}
\usepackage{tikz}
\usepackage{graphicx}
\usepackage[a4paper]{geometry}
\usepackage{color}
\usepackage{hyperref}
\numberwithin{equation}{section}

\theoremstyle{plain}
\newtheorem{thm}{Theorem}[section]

\newtheorem{cor}[thm]{Corollary}
\newtheorem{lem}[thm]{Lemma}
\theoremstyle{definition}

\newtheorem{rem}[thm]{Remark}

\newtheorem{exa}[thm]{Example}
\newtheorem*{rem*}{Remark}

\newcommand{\dd}{\mathrm{d}}
\newcommand{\ee}{\mathrm{e}}
\newcommand{\ii}{\mathrm{i}}

\newcommand{\To}{\longrightarrow}
\renewcommand{\Im}{\operatorname{Im}}
\renewcommand{\Re}{\operatorname{Re}}

\newcommand{\Z}{\mathbb{Z}}

\newcommand{\R}{\mathbb{R}}
\newcommand{\C}{\mathbb{C}}

\newcommand{\T}{\mathbb{T}}

\newcommand{\cA}{\mathcal{A}}

\newcommand{\eps}{\epsilon}
\newcommand{\veps}{\varepsilon}

\def\ds{\displaystyle}

\DeclareMathOperator{\sgn}{sgn}

\DeclareSymbolFont{extraup}{U}{zavm}{m}{n}
\DeclareMathSymbol{\varheart}{\mathalpha}{extraup}{86}
\DeclareMathSymbol{\vardiamond}{\mathalpha}{extraup}{87}

\title{Dispersion for Schr\"odinger operators on regular trees}
\author{Ka\"is Ammari}
\address{UR Analysis and Control of PDEs, UR13ES64, Department of Mathematics, Faculty of Sciences of
Monastir, University of Monastir, 5019 Monastir, Tunisia.}
\email{kais.ammari@fsm.rnu.tn}
\author{Mostafa Sabri}
\address{Department of Mathematics, Faculty of Science, Cairo University, Giza 12613, Egypt.}
\email{mmsabri@sci.cu.edu.eg}
\usepackage{calc}
\usepackage{graphicx}

\makeatletter
\newlength{\temp@wc@width}
\newlength{\temp@wc@height}
\newcommand{\widecheck}[1]{%
  \setlength{\temp@wc@width}{\widthof{$#1$}}%
  \setlength{\temp@wc@height}{\heightof{$#1$}}%
  #1\hspace{-\temp@wc@width}%
  \raisebox{\temp@wc@height+2pt}[\heightof{$\widehat{#1}$}]%
     {\rotatebox[origin=c]{180}{\vbox to 0pt{\hbox{$\widehat{\hphantom{#1}}$}}}}%
}
\makeatother

\begin{document}

\begin{abstract}
We prove dispersive estimates for two models~: the adjacency matrix on a discrete regular tree, and the Schr\"odinger equation on a metric regular tree with the same potential on each edge/vertex. The latter model can be thought of as an extension of the case of periodic Schr\"odinger operators on the real line. We establish a $t^{-3/2}$-decay for both models which is sharp, as we give the first-order asymptotics.
\end{abstract}

\maketitle

\tableofcontents

\section{Introduction}

This article is concerned with the analysis of the Schwartz kernel of the evolution operator $\ee^{\ii tH}$ associated to a Schr\"odinger operator $H$ acting on the $(q+1)$-regular tree. More precisely, we provide first-order asymptotics for the kernel as $t\To+\infty$, which imply sharp dispersive estimates with decay $\asymp t^{-3/2}$ for any $q$. Our main interest is in quantum graphs, that is, the continuum case in which each edge of the tree is endowed a length and a potential, and the operator $H$ acts as a one-dimensional differential operator on the edges. This model can be regarded as an extension of the case of Schr\"odinger operators with a periodic potential on $\R$. However the proof is a lot simpler for the discrete model, so we treat it as well to introduce some key ideas.

\medskip

Consider the adjacency matrix on $\Z$, $(\cA f)(j) = f(j-1)+f(j+1)$ for $f\in\ell^2(\Z)$. An easy application of the Fourier transform $\mathscr{F}:L^2[0,2\pi]\to \ell^2(\Z)$, $f\mapsto (\hat{f}_k)$, shows that its spectrum $\sigma(\cA)=\sigma_{ac}(\cA) = [-2,2]$. In fact $\cA$ is seen to be unitarily equivalent to the operator of multiplication by $\phi(x)=2\cos x$ on $L^2[0,2\pi]$. It also follows that
\[
\ee^{\ii t\cA}(n,m) = \langle \mathscr{F}^{-1}\delta_n,(\mathscr{F}^{-1}\ee^{\ii t\cA}\mathscr{F})\mathscr{F}^{-1}\delta_m\rangle = \frac{1}{2\pi}\int_{0}^{2\pi}\ee^{\ii(m-n)x}\ee^{\ii t\phi(x)}\,\dd x = \ii^{m-n} J_{m-n}(2t)\,,
\]
where $J_k$ is the Bessel function. Its asymptotics are well-known as $t\To\infty$. For example, using \cite[p. 338]{Stein}, we deduce that
\begin{equation}\label{e:besasym}
\ee^{\ii t\cA}(n,m) = \ii^{m-n}\sqrt{\frac{2}{\pi}}\cdot \frac{1}{\sqrt{2t}}\cos\Big(2t-\frac{\pi(m-n)}{2}-\frac{\pi}{4}\Big)+O\Big(\frac{1}{t^{3/2}}\Big)\,.
\end{equation}

Since
\begin{equation}\label{e:disperz}
\|\ee^{\ii t\cA} f\|_{\infty} = \sup_{n\in\Z}\bigg|\sum_{m\in \Z}\ee^{\ii t \cA}(n,m)f(m)\bigg|\le \sup_{n,m\in \Z} |\ee^{\ii t \cA}(n,m)|\cdot \|f\|_1\,,
\end{equation}
one may be tempted to deduce that $\|\ee^{\ii t\cA} f\|_{\infty} \le \frac{C}{t^{1/2}} \|f\|_1$ for large $t$. However, the remainder term $O(t^{-3/2})$ depends on $m-n$, we actually have $\sup_{n,m}O(t^{-3/2})=\infty$ in the above asymptotics \cite{Lan,Olen,Kra}. The correct uniform bound that one can get is $\|\ee^{\ii t\cA} f\|_{\infty} \le \frac{C}{t^{1/3}}\|f\|_1$, see \cite{Lan,SteKev}. More generally, for $\cA$ on $\ell^2(\Z^d)$, we have $\ee^{\ii t \cA}(n,m) = \prod_{j=1}^d \ii^{m_j-n_j}J_{m_j-n_j}(2t)$, and one may deduce for large $t$ the \emph{sharp} estimate \cite{SteKev},
\[
\|\ee^{\ii t\cA} f\|_{\infty}\le \frac{C}{t^{d/3}}\|f\|_1\,.
\]

Since $\|\ee^{\ii t \cA} f\|_2 = \|f\|_2$ is constant with time, the fact that the kernel decays with $t$ shows that wave packets must be spreading out as time goes on. It is then quite intuitive that in higher dimensions there is more spreading out as there are more directions to diffuse.

Back to \eqref{e:besasym}, we emphasize that throughout the paper we are interested in the large $t$ behavior. Another interesting asymptotics is to fix time and study the kernel $\ee^{\ii t \cA}(x,y)$ as $d(x,y)\To\infty$. In case of $\Z$, it is known that the Bessel function $J_k(2t) \sim \frac{1}{\sqrt{2\pi k}}(\frac{\ee t}{k})^k$ as $k\To\infty$, see \cite[9.3.1, p.365]{Abra}. This says that $|(\ee^{\ii t \cA}\delta_0)(n)|=|\langle \delta_n, \ee^{\ii t\cA} \delta_0\rangle| = |J_n(2t)|$ decays very fast in $n$. Such decay cannot be inferred from asymptotics of the form \eqref{e:besasym} (and conversely $J_k(2t) \sim \frac{1}{\sqrt{2\pi k}}(\frac{\ee t}{k})^k$ is of course not the correct asymptotics as $t\To\infty$). To summarize, the order of limits ($t$ or $d(x,y)$) is important.

\medskip

We now explore the case of the infinite $(q+1)$-regular tree $\T_q$ with $q\ge 2$. Recall that $\T_q$ is a connected graph with no cycles such that each vertex has $(q+1)$ neighbors. We will show that here again, as expected, the spreading is faster than on $\Z$.

Here $(\cA f)(v)= \ds \sum_{w\sim v}f(w)$, where the sum is over nearest neighbors $w$ of $v$. It is well-known that $\sigma(\cA) = \sigma_{ac}(\cA) = [-2\sqrt{q},2\sqrt{q}]$, see e.g. \cite{CdV98}.

\begin{thm}\label{thm:maincomb}
Consider the adjacency matrix $\cA$ on the $(q+1)$-regular tree $\T_q$, $q\ge 2$.

The following asymptotics hold for the evolution semigroup~: for any vertices $v,w\in \T_q$ with $d(v,w)=n$, as $t\To \infty$, we have
\[
\ee^{\ii t\cA}(v,w) = \begin{cases} \frac{1}{\sqrt{\pi}\,t^{3/2}}\sin\left(2\sqrt{q}t-\frac{\pi}{4}\right)\cdot \frac{q^{\frac{-n}{2}+\frac{1}{4}}(2+(n+1)(q-1))}{(q-1)^2} +O(t^{-2})& \text{if } n \text{ is even},\\ \frac{-\ii}{\sqrt{\pi}\,t^{3/2}}\sin\left(2\sqrt{q}t+\frac{\pi}{4}\right)\cdot \frac{q^{\frac{-n}{2}+\frac{1}{4}}(2+(n+1)(q-1))}{(q-1)^2} +O(t^{-2}) & \text{if } n \text{ is odd.}\end{cases}
\]
The term $O(t^{-2})$ is uniform in $n=d(v,w)$ and independent of $v,w$. In particular, we may find $t_0$ and $C_q$ such that for all $t>t_0$, all $f\in L^1(\T_q)$, we have
\begin{equation}\label{e:maindis}
\|\ee^{\ii t\cA} f\|_{\infty}\le \frac{C_q}{t^{3/2}}\|f\|_1\,.
\end{equation}
\end{thm} 

This result shows that the spreading is indeed faster than $\Z$~: we have the asymptotics $\asymp t^{-3/2}$ instead of $t^{-1/2}$, corresponding to the fact that the wave now has exponentially more directions to go to. Curiously however the effect of increasing the degree $q$ is not as substantial as the Euclidean case. While on $\Z^d$ one gets the asymptotics $\asymp t^{-d/2}$, here we always have $\asymp t^{-3/2}$, though increasing the degree does decrease the kernel $\ee^{\ii t\cA}(v,w)$.
%

\bigskip

We now consider the continuum case. The prototype here is the Euclidean space. Simple arguments using the Fourier transform \cite[Chapter 7.4]{Teschl} reveal that for the Laplacian on $L^2(\R^d)$, we have
\[
\ee^{\ii t\Delta}f(x) = \frac{1}{(4\pi\ii t)^{d/2}}\int_{\R^d}\ee^{\ii\frac{|x-y|^2}{4t}}\psi(y)\,\dd y\,.
\]
In particular, $\|\ee^{\ii t \Delta}\|_{L^1(\R^d)\to L^\infty(\R^d)} \le C t^{-d/2}$. In other words, the same type of decay we saw in the discrete model $\Z^d$ holds in the continuum, in fact it is a bit faster here. So let us consider the continuum analog of the regular tree, namely the infinite regular quantum tree.

Consider the $(q+1)$-regular tree $\T_q$ and endow it with an equilateral quantum structure: all edges have length $L$, carry the same potential $W$, and all vertices carry the same coupling constant $\alpha\in\R$. Assume the potential $W$ is edgewise symmetric, i.e. $W(L-x)=W(x)$. Denote this quantum tree by $\mathbf{T}_q$ and consider the corresponding Schr\"odinger operator
\[
H=-\Delta+W
\]
acting on the Hilbert space $L^2(\mathbf{T}_q) = \ds \mathop\oplus_{e\in E(\mathbf{T}_q)} L^2(0,L)$ with domain given by the set of functions $f=(f_e)\in \mathop\oplus H^2(0,L)$ satisfying the following $\delta$-boundary conditions: continuity at vertices, i.e. $f_e(v)=f_{e'}(v)=:f(v)$ if $e,e'$ are edges with origin $v$, and current relation $\ds \sum_{e,o(e)=v} f_e'(v) = \alpha f(v)$ for any vertex $v\in \T_q$. For $f=(f_e)\in D(H)$, we then have $Hf_e(x) = -f_e''(x)+W(x)f_e(x)$. The case $\alpha=0$ corresponds to the well-known Kirchhoff boundary conditions, where current is perfectly preserved.

We know by \cite{Carl} that $\sigma(H) = \{I_n\}_{n=1}^\infty \cup \{\delta_n\}_{n=1}^\infty$, where $I_n = [a_n,b_n]$ are bands of purely absolutely continuous spectrum (AC for short), and $\delta_n$ is an infinitely degenerate eigenvalue lying in the gap between $I_n$ and $I_{n+1}$. There is no singular continuous spectrum. We discuss the spectrum in more detail in Section~\ref{sec:prelimspe}.

We are now interested in the kernel of the evolution operator $\ee^{\ii tH}$. By the spectral theorem, if $\lambda$ is an eigenvalue with corresponding eigenvector $\psi_{\lambda}$, then $\ee^{\ii tH} \psi_{\lambda} = \ee^{\ii t \lambda}\psi_{\lambda}$. In particular $\|\ee^{\ii t H}\psi_{\lambda}\|_{\infty}=\|\psi_{\lambda}\|_{\infty}$ does not exhibit any decay with $t$. Therefore we should restrict our attention to the band spectrum, more precisely functions $\psi$ in the absolutely continuous subspace $\mathscr{H}_{ac}\subset\mathscr{H}=L^2(\mathbf{T}_q)$.

To state our main result, we need some notations. Given $z\in \C$, consider the eigenproblem on $[0,L]$
\[
-\psi_e''+W\psi_e=z\psi_e \,.
\]
Choose a basis of solutions $C_z(x),S_z(x)$ satisfying the initial values
\[
\begin{pmatrix} C_z(0)& S_z(0)\\C_z'(0)& S_z'(0)\end{pmatrix} = \begin{pmatrix} 1&0\\0&1\end{pmatrix}.
\]
Note that $C_z(x)=\cos\sqrt{z}x$ and $S_z(x)=\frac{\sin\sqrt{z}x}{\sqrt{z}}$ if $W\equiv 0$. Here, if $z\in\C^+:=\{z\in \C:\Im z>0\}$, we always choose the branch of $\sqrt{z}$ with positive imaginary part. Denote
\begin{equation}\label{e:cs}
c(z) = C_z(L), \qquad s(z) = S_z(L)
\end{equation}
and
\begin{equation}\label{e:muw}
w(z):= (q+1)c(z)+\alpha s(z)\,.
\end{equation}

Finally let $G^z_{\mathbf{T}_q}(x,y)$ be the Green's function (i.e. the resolvent kernel) of $H$, for $z\in \C^+$ and $x,y\in\mathbf{T}_q$. It is known that the limit $G^{\lambda+\ii 0}_{\mathbf{T}_q}(x,y)$ exists for $\lambda\in\sigma_{ac}$, see Section~\ref{sec:prelimspe}. 

\begin{thm}\label{thm:mainquan}
Consider the Schr\"odinger operator $H=-\Delta+W$ on the $(q+1)$-regular tree $\mathbf{T}_q$, $q\ge 2$, where $W$ is edge-symmetric and identical on each edge, and each vertex is endowed the same coupling constant $\alpha\in\R$. 

Denote by $I_n=[a_n,b_n]$ the $n$-th band of AC spectrum when the bands are arranged in increasing order. The following asymptotics hold for the evolution semigroup~: for any $x,y\in \mathbf{T}_q$, as $t\To \infty$, we have
\begin{multline}\label{e:mainkerquan}
\ee^{\ii tH}\mathbf{1}_{ac}(H)(x,y)= \frac{\ii q^{1/4}(q+1)}{(q-1)^2\sqrt{\pi}\, t^{3/2}}\sum_{n\ge 1} \bigg[\frac{\ee^{\frac{-\ii \pi}{4}}\ee^{\ii a_n t}|w'(a_n)|^{1/2}\left|s(a_n)\right|\Phi(a_n,x,y)}{2} \\
- \frac{\ee^{\frac{\ii \pi}{4}}\ee^{\ii b_n t}|w'(b_n)|^{1/2}\left|s(b_n)\right|\Phi(b_n,x,y)}{2}\bigg]+O(t^{-2})\,,
\end{multline}
where
\begin{equation}\label{e:Philxy}
\Phi(\lambda,x,y) = \frac{\Im G^{\lambda+\ii 0}_{\mathbf{T}_q}(x,y)}{\Im G^{\lambda+\ii 0}_{\mathbf{T}_q}(o,o)}\,,
\end{equation}
with $o\in \T_q$ an arbitrary vertex.
The term $O(t^{-2})$ is uniform in $d(x,y)$ and independent of $x,y$. In particular, we may find $t_0$ and $C_q$ such that for all $t>t_0$, all $f\in L^1(\mathbf{T}_q)\cap L^2(\mathbf{T}_q)$, we have
\begin{equation}\label{e:maindis2}
\|\ee^{\ii tH}\mathbf{1}_{ac}(H) f\|_{\infty}\le \frac{C_q}{t^{3/2}}\|f\|_1\,.
\end{equation}
\end{thm}

The function $\Phi(\lambda,x,y)$ represents the correlation of a generic wavefunction of $H$ on $\mathbf{T}_q$, with energy $\lambda$, at the points $x,y$. See \cite{ISW} for a precise statement. It can be written explicitly in terms of the functions $S_{\lambda}(x)$ defined before \eqref{e:cs}, see \eqref{e:psi2} and \eqref{e:psi3}. In particular, it is known that $\Phi(\lambda,x,y)$ is bounded over $\sigma_{ac}(H)$ and decays exponentially in $d(x,y)$. The function $s(\lambda)$ is given by \eqref{e:cs} and $w(\lambda)$ as in \eqref{e:muw}. We also denoted $w'(\lambda):=\partial_{\lambda}w(\lambda)$. The sum in \eqref{e:mainkerquan} is absolutely convergent, as follows from the estimates of Section~\ref{sec:quancase}.

The main term in \eqref{e:mainkerquan} depends on $\alpha$ through $a_n,b_n,w'$ and $\Phi$. As $\alpha\to\infty$, one approaches the Dirichlet conditions, where no dispersion is expected since the tree degenerates into a direct sum of segments with no continuous spectrum. This can also be seen in the main term; as discussed in \S~\ref{sec:prelimspe}, $\lambda\in \sigma_{ac}$ iff $|w(\lambda)|\le 2\sqrt{q}$. Since $w(\lambda) = (q+1)c(\lambda)+\alpha s(\lambda)$, we see that for $\lambda$ to stay in $\sigma_{ac}$ as $\alpha\to\infty$, we must have $s(\lambda)\to 0$. This implies that the bands $[a_n,b_n]$ move and shrink to the zero $\delta_n$ of $s(\lambda)$ in the limit, so the main term vanishes as $s(a_n)=s(b_n)=s(\delta_n)=0$.

On the other hand, if $W=\alpha=0$, the main term becomes a bit more explicit; we get $a_n = (\frac{(n-1)\pi+\theta}{L})^2$, $b_n=(\frac{n\pi-\theta}{L})^2$ and $|w'(\lambda)|^{1/2}|s(\lambda)| = \sqrt{\frac{L(q+1)}{2}}(\frac{|\sin\sqrt{\lambda}L|}{\sqrt{\lambda}})^{3/2}$. Here $\theta=\arccos \frac{2\sqrt{q}}{q+1}$. If $x,y$ belong to the same edge, then $\Phi(\lambda,x,y) = \cos(\sqrt{\lambda}(x-y))$ using the arguments in \cite[p.108]{ISW}.

Our result for the adjacency matrix on $\T_q$ also involves $\Phi$ actually. Namely, the quantity $q^{-n/2}(2+(n+1)(q-1))$ in Theorem~\ref{thm:maincomb} is nothing but $(q+1)\Phi_{\cA}(2\sqrt{q},v,w)$. 

\medskip

Note that since $\|\ee^{\ii t H}\mathbf{1}_{ac}\|_{L^2\to L^2}\le 1$, one may use the Riesz-Thorin interpolation to deduce from our results that $\|\ee^{\ii t H}\mathbf{1}_{ac}\|_{L^{p'}\to L^p}\le C t^{\frac{-3}{2}(1-\frac{2}{p})}$, for $p'=\frac{p}{p-1}$, $p\ge 2$.

\subsection{Previous results}
Besides the well-known case of the Laplacian on the real line, there are some interesting earlier works concerning dispersion on networks. In \cite{AAN}, the authors consider the Laplacian on the tadpole graph (this consists of the half-line with a loop at the origin, a.k.a. a lasso graph). It is found that the speed of dispersion in the AC subspace cannot exceed $t^{-1/2}$. In \cite{AAN2}, the case of star graphs is considered, where each edge in the star has infinite length (a half-line). The authors allow for the presence of potentials on the edges satisfying some decay condition, and establish the speed of dispersion $t^{-1/2}$. In \cite{BI}, the authors replace the central vertex of the star by a finite tree, in other words they study a finite tree with terminal leaves of infinite length. They only consider the Laplacian, i.e. no edge-potential, but they allow for coupling constants $(\alpha_j)_{j=1}^p$ at the vertices of the finite tree. They establish a speed of dispersion $t^{-1/2}$.

Though it is not always proved that the speed of dispersion is sharp in these models, one heuristically expects it to be the case, because one can think of these models as being essentially one-dimensional, with some geometric obstacle of finite size (of course this heuristics doesn't imply there is any simple proof for the asymptotics). This is not the case for the regular tree which is genuinely different, hence the distinction in the speed which seems to be observed for the first time in quantum graphs.

We finally mention that the dispersive estimates have classically found important applications in the study of nonlinear equations, see e.g. \cite{HMMS,T} concerning Strichartz estimates and \cite{Cucca2} for a perturbation problem.

\subsection{The case $q=1$}\label{sec:q=1}
The case (not considered in our paper) of $(q+1)$-regular trees with $q=1$ corresponds to the real line $\R$. The corresponding operator $H$ becomes a periodic Schr\"odinger operator on $\R$, with $L$-periodic potential $W$, and moreover a singular $\delta$-potential $\alpha  \ds \sum_{n\in\Z}\delta(t-nL)$ of Kronig-Penney type. If $W=\alpha=0$, this is just the Laplacian on $\R$, the speed of dispersion is then $t^{-1/2}$. If $\alpha=0$ and $W$ is non-zero, we have a standard periodic Schr\"odinger operator on $\R$. This case has already been studied in the literature \cite{Fir,Cai,Cucca}. It was observed that the presence of the potential can slow down the speed of dispersion to $t^{-1/3}$. Technically speaking, it seems the main reason for this is that the modulus and $\lambda$-variation of the quantity corresponding to \eqref{e:Philxy} does not decay in $d(x,y)$, in contrast to our case, where these decay as $\lesssim Cq^{-d(x,y)/2}$. Heuristically, it is indeed plausible that the wavefunction correlations between distinct points in the quantum tree are much less pronounced than in one dimension. For a more detailed comparison with periodic Schr\"odinger operators on $\R$, see Appendix~\ref{app:per}.

\subsection{Open problems}

To the best of our knowledge, the study of dispersion for periodic Schr\"odinger operators on $\R^d$ in dimension $d\ge 2$ has not been addressed yet. Is the decay still $\lesssim t^{-d/2}$ as the free case, or can it slow down like in $d=1$~? The Bethe-Sommerfeld conjecture, now a theorem \cite{Par,Vel}, may provide a simplifying input for the analysis. In fact, in one dimension, dealing with infinitely many spectral bands \cite{Cucca} is a lot more technical than dealing with finitely many bands \cite{Fir}. This difficulty also appears in our analysis of $\mathbf{T}_q$, $q\ge 2$, in the control of infinite series of error terms. 

In a different direction, it is natural to ask about dispersion in more general trees. A quite large family of trees with a somewhat periodic geometry and potential is given by universal covers of finite graphs. Such trees have attracted a lot of interest in recent years \cite{KLW,AS2,ABS,AISW}. It is now known both in the discrete and metric case that the spectrum of these trees consists of bands of purely absolutely continuous spectrum with possibly some degenerate eigenvalues. Our present proof of dispersion however uses somewhat more precise information on the Green's functions which is not yet available for these more general trees.

\begin{rem}
After the paper was submitted to the arXiv, Maxime Ingremeau informed us of the paper of A.~J.~Eddine \emph{``Schr\"odinger equation on homogeneous trees''} J. Lie Theory \textbf{23} (2013), 779--794. This paper also provides dispersive estimates for the discrete model $(\cA,\T_q)$. Therein, we also discovered the paper of A.~G.~Setti \emph{``$L^p$ and operator norm estimates for the complex time heat operator on homogeneous trees''}, Trans. Amer. Math. Soc. \textbf{350} (1998), 743--768, which also proves dispersive estimates for $(\cA,\T_q)$. These proofs are different from ours, relying on harmonic analysis on trees rather than Green's functions, and the papers do not treat the continuum model of quantum graphs, which is our main concern as it offers new challenges, as can already be seen in the transition from $(\cA,\Z)$ to periodic Schr\"odinger operators $(H,\R)$.
\end{rem}

\section{Proof of the result for combinatorial trees}\label{sec:comb}
%
This section is devoted to the proof of Theorem~\ref{thm:maincomb}. The evolution kernel is essentially the Fourier transform of the spectral density, see \eqref{e:specthm}. On the other hand, the spectral density at different vertices is a multiple of the one on the diagonal, see \eqref{e:imgvw}. We thus begin by proving Theorem~\ref{thm:maincomb} for $v=w$ using the stationary phase lemma. We use a version with an explicit error given in Appendix~\ref{app:stat}. This is especially useful to ensure that the estimate is uniform in the distance $n$ between $v$ and $w$ in the general case.

\subsection{Basic considerations} 
A well-known analysis of the resolvent \cite{KS,Klein,AS} shows that for $\lambda\in\sigma(\cA)=[-2\sqrt{q},2\sqrt{q}]$,
\begin{equation}\label{e:discregree}
\Im G^{\lambda+\ii 0}(v,v) = \frac{(q+1)\sqrt{4q-\lambda^2}}{2[(q+1)^2-\lambda^2]}\,.
\end{equation}

More generally, if $d(v,w)=n$, then
\begin{equation}\label{e:imgvw}
\Im G^{\lambda+\ii 0 }(v,w) = \Im G^{\lambda + \ii 0}(v,v) \cdot \Phi_n(\lambda)\,,
\end{equation}
where $\Phi_{n}(\lambda)$ is the \emph{spherical function of $\T_q$}, cf. \cite{ISW}. It is given explicitly by
\begin{equation}\label{e:sfer}
\Phi_n(\lambda)=q^{-n/2}\bigg(\frac{2}{q+1} P_n\Big(\frac{\lambda}{2\sqrt{q}}\Big) + \frac{q-1}{q+1}Q_n\Big(\frac{\lambda}{2\sqrt{q}}\Big)\bigg)\,,
\end{equation}
where $P_n(\cos\theta)=\cos n\theta$ and $Q_n(\cos\theta)=\frac{\sin(n+1)\theta}{\sin\theta}$ are the Chebyshev polynomials of the first and second kinds, respectively.

It follows by the spectral theorem (see \cite[Lemma 3.6]{AS}) that if $\mu_{v,w}$ is the spectral measure at $\delta_v,\delta_w$, then
\begin{equation}\label{e:specthm}
\ee^{\ii t\cA}(v,w) = \int_{\sigma(\cA)} \ee^{\ii t\lambda}\,\dd\mu_{v,w}(\lambda) = \frac{1}{\pi}\int_{-2\sqrt{q}}^{2\sqrt{q}}\ee^{\ii t\lambda}\Im G^{\lambda+\ii 0}(v,w)\,\dd\lambda\,.
\end{equation}

\subsection{Diagonal terms}
In view of \eqref{e:imgvw} and \eqref{e:specthm}, we first prove the theorem for $v=w$.

Let $\Psi(\lambda) = \Im G^{\lambda+\ii 0}(v,v) = \frac{(q+1)\sqrt{4q-\lambda^2}}{2[(q+1)^2-\lambda^2]}$. Then $\Psi(\pm 2\sqrt{q})=0$, so integrating by parts,
\begin{equation}\label{e:1stbypar}
\int_{-2\sqrt{q}}^{2\sqrt{q}}\ee^{\ii t\lambda}\Psi(\lambda)\,\dd\lambda = \frac{\ee^{\ii t\lambda}}{\ii t}\Psi(\lambda)\Big|_{-2\sqrt{q}}^{2\sqrt{q}} - \frac{1}{\ii t}\int_{-2\sqrt{q}}^{2\sqrt{q}}\ee^{\ii t\lambda}\Psi'(\lambda)\,\dd\lambda = \frac{\ii}{t}\int_{-2\sqrt{q}}^{2\sqrt{q}}\ee^{\ii t\lambda}\Psi'(\lambda)\,\dd\lambda\,.
\end{equation}

On the other hand,
\begin{equation}\label{e:psi'}
\Psi'(\lambda)
=\frac{\lambda(q+1)(6q-\lambda^2-q^2-1)}{2((q+1)^2-\lambda^2)^2\sqrt{4q-\lambda^2}}\,.
\end{equation}

This has a singularity at $\pm 2\sqrt{q}$. But note that $\Psi'(\lambda)=0$ iff $\lambda=0$ or $\lambda^2=6q-q^2-1$. For $q\ge 6$ the latter case never occurs. In fact, we see more precisely that for $q\ge 6$, $\Psi'(\lambda)$ is positive on $[-2\sqrt{q},0]$ and negative on $[0,2\sqrt{q}]$. Hence,
\begin{multline}\label{e:totvar}
|\ee^{\ii t\cA}(v,v)|\le \frac{1}{\pi t}\int_{-2\sqrt{q}}^{2\sqrt{q}}|\Psi'(\lambda)|\,\dd\lambda =\frac{1}{\pi t}\Big(\int_{-2\sqrt{q}}^{0}\Psi'(\lambda)\,\dd\lambda -\int_{0}^{2\sqrt{q}}\Psi'(\lambda)\,\dd\lambda\Big)\\
= \frac{2\Psi(0)}{\pi t} = \frac{2\sqrt{q}}{(q+1)\pi t}\,.
\end{multline}

The case $2\le q\le 5$ can be handled similarly, here $\Psi'$ has additional sign changes at $E_{\pm}:= \pm\sqrt{6q-q^2-1}$, so the bound becomes 
$$|\ee^{\ii t\cA}(v,v)|\le \frac{2\Psi(E_-)-2\Psi(0)+2\Psi(E_+)}{\pi t}=\frac{1}{\pi t}\left(\frac{q+1}{q-1}-\frac{2\sqrt{q}}{q+1}\right).$$ 

Still, this bound on $\ee^{\ii t\cA}(v,v)$ is based on the trivial upper bound $|\int \ee^{\ii t\lambda}\Psi'(\lambda)\,\dd\lambda|\le \int |\Psi'(\lambda)|\,\dd\lambda$, which is lossy. To estimate the speed of decay more carefully we shall use the \emph{method of stationary phase} for such oscillatory integrals.

To avoid the singular expression of $\Psi'(\lambda)$, we first consider the change of variables $\lambda = 2\sqrt{q}\cos\theta$. This gives
\[
\int_{-2\sqrt{q}}^{2\sqrt{q}}\ee^{\ii t\lambda}\Psi'(\lambda)\,\dd\lambda = \int_0^\pi \ee^{2\ii t\sqrt{q}\cos\theta}\frac{\sqrt{q}\cos\theta(q+1)(6q-4q\cos^2\theta-q^2-1)}{((q+1)^2-4q\cos^2\theta)^2}\,\dd\theta\,.
\]
The phase function $\phi(\theta)=2\sqrt{q}\cos\theta$ has two critical points $0,\pi$, at which $\phi'(x)=0$. In principle such an integral can now be controlled using standard stationary phase results \cite[p.334]{Stein} or \cite[Theorem 3.11]{Zwor}, after multiplying by bump functions around the critical points and controlling the remainder errors. Though this is enough for the present integral, in the later case of quantum graphs we shall need the more precise version given in Corollary~\ref{cor:phase}, which gives an explicit bound on the error. So let us illustrate its use here; it also has the advantage of being directly applicable to critical points on the boundaries of finite intervals, as we have here.

\medskip

We define
\begin{equation}\label{e:g}
g(\theta)=\frac{\sqrt{q}\cos\theta(q+1)(6q-4q\cos^2\theta-q^2-1)}{((q+1)^2-4q\cos^2\theta)^2}\,
\end{equation}
and divide $\int_0^\pi = \int_0^{\pi/2} + \int_{\pi/2}^\pi$. For the first part, the only critical point is at $0$. Moreover, $\phi''(0) = -2\sqrt{q}$. So \eqref{e:q11} takes the form $Q_{1,1}(\theta) = \frac{g(\theta)}{2\sqrt{q}\sin\theta}-\frac{g(0)}{2\sqrt{2q}\sqrt{1-\cos\theta}}$. Using $1-\cos\theta=2\sin^2\frac{\theta}{2}$ and $g(0)=\frac{-\sqrt{q}(q+1)}{(q-1)^2}$ we get
\begin{multline*}
Q_{1,1}(\theta)=\frac{q+1}{2\sin\theta}\Big[\frac{\cos\theta(6q-4q\cos^2\theta-q^2-1)}{((q+1)^2-4q\cos^2\theta)^2}+\frac{\cos\frac{\theta}{2}}{(q-1)^2}\Big]\\
=\frac{q+1}{2\sin\theta}\Big[\frac{\cos\theta[-(q-1)^2+4q\sin^2\theta](q-1)^2+\cos\frac{\theta}{2}[(q-1)^2+4q\sin^2\theta]^2}{((q-1)^2+4q\sin^2\theta)^2(q-1)^2}\Big]\\
=\frac{q+1}{2\sin\theta}\Big[\frac{(q-1)^4(\cos\frac{\theta}{2}-\cos\theta)+4q(q-1)^2\sin^2\theta[\cos\theta+2\cos\frac{\theta}{2}]+16q^2\sin^4\theta\cos\frac{\theta}{2}}{((q-1)^2+4q\sin^2\theta)^2(q-1)^2}\Big].
\end{multline*}
We already know from \eqref{e:q11a} that $Q_{1,1}(0)$ is finite, but it is useful to have a uniform bound. The only term we should control above is $\frac{\cos\frac{\theta}{2}-\cos\theta}{\sin\theta}$. We write $\frac{\cos\frac{\theta}{2}-\cos\theta}{\sin\theta} = \frac{\cos\frac{\theta}{2}-\cos^2\frac{\theta}{2}+\sin^2\frac{\theta}{2}}{2\sin\frac{\theta}{2}\cos\frac{\theta}{2}} = \frac{1-\cos\frac{\theta}{2}}{2\sin\frac{\theta}{2}}+\frac{\sin\frac{\theta}{2}}{2\cos\frac{\theta}{2}} = \frac{\sin\frac{\theta}{4}}{2\cos\frac{\theta}{4}}+\frac{\sin\frac{\theta}{2}}{2\cos\frac{\theta}{2}}$. Thus,
\begin{multline}\label{e:q11comb}
Q_{1,1}(\theta) = \frac{q+1}{((q-1)^2+4q\sin^2\theta)^2(q-1)^2}\bigg[\frac{(q-1)^4}{4}\Big(\tan\frac{\theta}{4}+\tan\frac{\theta}{2}\Big)\\
+2q(q-1)^2\sin\theta\Big[\cos\theta+2\cos\frac{\theta}{2}\Big]+8q^2\sin^3\theta\cos\frac{\theta}{2}\bigg]\,.
\end{multline} 
As we saw in \eqref{e:totvar}, the total variation $V(f)=\int_0^{\pi/2} |f'|$ may be bounded by $2C_f\|f\|_\infty$, where $C_f$ is the number of roots of $f'$. It is clear from \eqref{e:q11comb} that $Q_{1,1}$ is analytic over $[0,\frac{\pi}{2}]$, so the number of roots of $Q_{1,1}'$ in $[0,\frac{\pi}{2}]$ is a finite number $C_Q$ which may be found exactly, and $V_{0,\frac{\pi}{2}}(Q_{1,1})\le 2C_Q\frac{q+1}{(q-1)^6}((q-1)^4+5q(q-1)^2+8q^2)=:2C_{Q}c_q$. Recalling $g(0)=\frac{-\sqrt{q}(q+1)}{(q-1)^2}$, it follows from Corollary~\ref{cor:phase} that
\[
\bigg|\int_0^{\pi/2}\ee^{\ii t\phi(\theta)}g(\theta)\,\dd \theta - \ee^{2\ii t \sqrt{q}}\ee^{-\pi\ii/4}\sqrt{\frac{\pi}{4\sqrt{q}t}}g(0)\bigg| \\ \le \frac{1}{t}\left(c_q + 2C_Qc_q + c_q'\right) ,
\]
where $c_q'=\frac{1}{\sqrt{2}}\frac{(q+1)}{(q-1)^2}$. 

We argue similarly for $\int_{\pi/2}^{\pi}$, where the critical point is at $\pi$, cf. \eqref{e:end}.

We conclude that
\begin{equation}\label{e:statpha}
\int_{-2\sqrt{q}}^{2\sqrt{q}}\ee^{\ii t\lambda}\Psi'(\lambda)\,\dd\lambda = \Big(\frac{\pi}{\sqrt{q}\,t}\Big)^{1/2}\cdot \frac{\ee^{\frac{-\ii \pi}{4}}\ee^{2\ii \sqrt{q} t}g(0) + \ee^{\frac{\ii \pi}{4}}\ee^{-2\ii \sqrt{q} t}g(\pi)}{2} + O(t^{-1})\,.
\end{equation}
Recalling \eqref{e:g}, we have $g(\pi)=-g(0)$. So this simplifies to
\[
\int_{-2\sqrt{q}}^{2\sqrt{q}}\ee^{\ii t\lambda}\Psi'(\lambda)\,\dd\lambda = \ii\Big(\frac{\pi}{\sqrt{q}\,t}\Big)^{1/2}\sin\Big(2\sqrt{q}t-\frac{\pi}{4}\Big)\cdot \frac{-\sqrt{q}(q+1)}{(q-1)^2} +O(t^{-1})\,.
\]

Using \eqref{e:specthm} and \eqref{e:1stbypar}, we finally conclude that
\[
\ee^{\ii t\cA}(v,v) = \frac{1}{t^{3/2}\sqrt{\pi}}\sin\Big(2\sqrt{q}t-\frac{\pi}{4}\Big)\cdot \frac{q^{1/4}(q+1)}{(q-1)^2} +O(t^{-2})\,,
\]
with $O(t^{-2})$ independent of $v$. This proves Theorem~\ref{thm:maincomb} for $v=w$.

\subsection{General case}
Now suppose that $d(v,w)=n$. Then by \eqref{e:imgvw},
\[
\Im G^{\lambda}(v,w) =\Psi(\lambda)\Phi_n(\lambda)\,.
\]
This quantity still vanishes at $\pm 2\sqrt{q}$, so we get as in \eqref{e:1stbypar},
\begin{equation}\label{e:genkerA}
\ee^{\ii t\cA}(v,w) = \frac{\ii}{\pi t}\int_{-2\sqrt{q}}^{2\sqrt{q}}\ee^{\ii t\lambda} \left[\Psi'(\lambda)\Phi_n(\lambda) + \Psi(\lambda)\Phi_n'(\lambda)\right]\dd\lambda\,.
\end{equation}

As we shall see, the first term can be handled in essentially the same way as before, so let us consider the second term which actually gives a weaker contribution. Since $\Psi(\lambda)\Phi_n'(\lambda)$ vanishes at $\pm 2\sqrt{q}$, then integrating by parts we have
\begin{equation}\label{e:errorder}
\int_{-2\sqrt{q}}^{2\sqrt{q}}\ee^{\ii t\lambda}\Psi(\lambda)\Phi_n'(\lambda)\,\dd \lambda = \frac{-1}{\ii t}\int_{-2\sqrt{q}}^{2\sqrt{q}}\ee^{\ii t\lambda}\{\Psi(\lambda)\Phi_n'(\lambda)\}'\,\dd\lambda\,.
\end{equation}

Recalling \eqref{e:psi'} we see that $\Psi'(\lambda) = \frac{\lambda(6q-q^2-1-\lambda^2)}{((q+1)^2-\lambda^2)(4q-\lambda^2)}\Psi(\lambda)$. Hence,
\[
\{\Psi(\lambda)\Phi_n'(\lambda)\}' = \Psi(\lambda)\left(\frac{\lambda(6q-q^2-1-\lambda^2)}{((q+1)^2-\lambda^2)(4q-\lambda^2)}\Phi_n'(\lambda)+\Phi_n''(\lambda)\right)\,.
\]
Since $\Phi_n(\lambda)$ is a polynomial of degree $n$ in $\lambda$ and $\Psi(\lambda)$ does not vanish for $|\lambda|<2\sqrt{q}$, we see that $\{\Psi(\lambda)\Phi_n'(\lambda)\}'$ has the same roots as a polynomial of degree $\le n+2$. Consequently the total variation $\int_{-2\sqrt{q}}^{2\sqrt{q}} |\{\Psi(\lambda)\Phi_n'(\lambda)\}'|\,\dd\lambda \le 2(n+2)\|\Psi \Phi_n'\|_{\infty}$. Clearly $|\Psi(\lambda)|\le \frac{(q+1)\sqrt{q}}{(q-1)^2}$. For $\Phi_n'$, we have
\[
\Phi_n'(\lambda) = \frac{q^{-n/2}}{2\sqrt{q}(q+1)}\bigg(2P_n'\Big(\frac{\lambda}{2\sqrt{q}}\Big)+(q-1)Q_n'\Big(\frac{\lambda}{2\sqrt{q}}\Big)\bigg)\,.
\]

On the other hand, over $[-1,1]$, we have $|P_n(x)|\le 1$. This follows by definition $P_n(x)=\cos(n\arccos x)$. Using the relation
\begin{equation}\label{e:cheby2}
Q_n=\begin{cases} 2 \ds \sum_{j=0}^m P_{2j+1} & \text{if } n=2m+1,\\ 2 \ds \sum_{j=0}^m P_{2j} -1 &\text{if } n=2m, \end{cases}
\end{equation}
we deduce that $|Q_n(x)|\le n+1$. The bounds are attained at $x=1$~: we have $P_n(1)=1$ and $Q_n(1)=n+1$.

We now use the classic identity $P_n'(x)=nQ_{n-1}(x)$. This implies $|P_n'(x)| \le n^2$. Using \eqref{e:cheby2}, we deduce that $|Q_n'(x)|\le (n+1)n^2$.

We thus conclude that $|\Phi_n'(\lambda)|\le \frac{n^2q^{-n/2}}{2\sqrt{q}(q+1)}\left(2+(n+1)(q-1)\right)$.

Gathering the estimates, we have shown that
\[
\int_{-2\sqrt{q}}^{2\sqrt{q}}|\{\Psi(\lambda)\Phi_n'(\lambda)\}'|\,\dd\lambda \le 2(n+2)\frac{n^2q^{-n/2}}{2(q-1)^2}\left(2+(n+1)(q-1)\right).
\]
This clearly vanishes as $n\To\infty$, so it is uniformly bounded by some $C_q$ for all $n$. Hence, recalling \eqref{e:errorder}, we have
\begin{equation}\label{e:errorder2}
\int_{-2\sqrt{q}}^{2\sqrt{q}}\ee^{\ii t\lambda}\Psi(\lambda)\Phi_n'(\lambda)\,\dd \lambda = O(t^{-1})
\end{equation}
uniformly in $n$.

Let us finally turn back to \eqref{e:genkerA}. It remains to control $\int_{-2\sqrt{q}}^{2\sqrt{q}}\ee^{\ii t\lambda}\Psi'(\lambda)\Phi_n(\lambda)\,\dd\lambda$. For this, as in the case of $w=v$, we consider the change of variables $\lambda=2\sqrt{q}\cos\theta$. In view of \eqref{e:psi'}, this gives
\[
\int_{-2\sqrt{q}}^{2\sqrt{q}}\ee^{\ii t\lambda}\Psi'(\lambda)\Phi_n(\lambda)\,\dd\lambda = \int_0^\pi\ee^{2\ii t\sqrt{q}\cos\theta} g(\theta)\Phi_n(2\sqrt{q}\cos\theta)\,\dd\theta
\]
with $g(\theta)$ defined by \eqref{e:g}.
We estimate the oscillatory integral exactly as before using the method of stationary phase. Again the phase function $\phi(\theta)=2\sqrt{q}\cos\theta$ only has critical points at $0,\pi$. We conclude as in \eqref{e:statpha} that
\begin{multline}\label{e:red}
\int_{-2\sqrt{q}}^{2\sqrt{q}}\ee^{\ii t\lambda}\Psi'(\lambda)\Phi_n(\lambda)\,\dd\lambda \\
= \Big(\frac{\pi}{\sqrt{q}\,t}\Big)^{1/2}\cdot \frac{\ee^{\frac{-\ii \pi}{4}}\ee^{2\ii \sqrt{q} t}g(0)\Phi_n(2\sqrt{q}\cos 0) + \ee^{\frac{\ii \pi}{4}}\ee^{-2\ii \sqrt{q} t}g(\pi)\Phi_n(2\sqrt{q}\cos\pi)}{2} + O(t^{-1})\,.
\end{multline}

The term $O(t^{-1})$ is uniform in $n$. Indeed, consider the error from Corollary~\ref{cor:phase} at $\theta=0$; the error at $\pi$ is similar. First, $\frac{2|q(a)|}{\sqrt{2|p''(a)|}\sqrt{|p(b)-p(a)|}}=\frac{|g(0)\Phi_n(2\sqrt{q})|}{\sqrt{2q}}\le \tilde{c}_q nq^{-n/2}$ which is uniformly bounded in $n$. Next, $Q_{1,1}(\theta) = \frac{g(\theta)\Phi_n(2\sqrt{q}\cos\theta)}{2\sqrt{q}\sin\theta}-\frac{g(0)\Phi_n(2\sqrt{q})}{2\sqrt{2q}\sqrt{1-\cos\theta}}$. Simplifying as before yields
\begin{multline}\label{e:q11gen}
Q_{1,1}(\theta)=\frac{q+1}{((q-1)^2+4q\sin^2\theta)^2(q-1)^2}\Big[\frac{(q-1)^4[\Phi_n(2\sqrt{q})\cos\frac{\theta}{2}-\Phi_n(2\sqrt{q}\cos\theta)\cos\theta]}{2\sin\theta}\\
+ 2q(q-1)^2\sin\theta\big[\Phi_n(2\sqrt{q}\cos\theta)\cos\theta+2\Phi_n(2\sqrt{q})\cos\frac{\theta}{2}\big]+8\Phi_n(2\sqrt{q})q^2\sin^3\theta\cos\frac{\theta}{2}\Big].
\end{multline}

As shown before \eqref{e:q11comb}, we have $\frac{\cos\frac{\theta}{2}-\cos\theta}{\sin\theta}=\frac{\tan\frac{\theta}{4}+\tan\frac{\theta}{2}}{2}$. Hence,
\begin{multline}\label{e:simp}
\frac{\Phi_n(2\sqrt{q})\cos\frac{\theta}{2}-\Phi_n(2\sqrt{q}\cos\theta)\cos\theta}{2\sin\theta} \\
= \frac{\Phi_n(2\sqrt{q})}{4}\Big(\tan\frac{\theta}{4}+\tan\frac{\theta}{2}\Big)-\cos\theta\frac{\Phi_n(2\sqrt{q}\cos\theta)-\Phi_n(2\sqrt{q})}{2\sin\theta}\,.
\end{multline}

On the other hand, $\frac{\Phi_n(2\sqrt{q}\cos\theta)-\Phi_n(2\sqrt{q})}{\sin\theta}= \frac{\theta \Phi_n'(2\sqrt{q}\cos\theta_1)(-2\sqrt{q}\sin\theta_1)}{2\sin\theta}$ for some $\theta_1\in (0,\theta)$.

Gathering \eqref{e:q11gen} and \eqref{e:simp}, we see that $Q_{1,1}(0)=0$ and $|Q_{1,1}(\frac{\pi}{2})|\le  C_qp(n)q^{-n/2}$, where $p(n)$ is polynomial in $n$. To control $V_{0,\frac{\pi}{2}}(Q_{1,1})$, it suffices to show that the derivative of \eqref{e:simp} is similarly bounded by some $\tilde{C}_q\tilde{p}(n)q^{-n/2}$ (by \eqref{e:q11gen} and our previous bounds on $\Phi_n,\Phi_n'$). In turn, we should estimate $(\frac{\Phi_n(2\sqrt{q}\cos\theta)-\Phi_n(2\sqrt{q})}{\sin\theta})' = -2\sqrt{q}\Phi_n'(2\sqrt{q}\cos\theta) - \frac{\cos\theta[\Phi_n(2\sqrt{q}\cos\theta)-\Phi_n(2\sqrt{q})]}{\sin^2\theta}$. This is indeed bounded as desired (for the latter term note that $\frac{\theta\sin\theta_1}{\sin^2\theta}\To 1$ as $\theta\downarrow 0$).

We have shown the error from Corollary~\ref{cor:phase} takes the form $t^{-1}c_qp(n)q^{-n/2}$ for some polynomial $p(n)$, it is thus $O(t^{-1})$ uniformly in $n$ as required. 

\medskip

Back to \eqref{e:red}, we have $g(\pi)=-g(0)$ while $\Phi_n(-2\sqrt{q}) = (-1)^n\Phi_n(2\sqrt{q})$. Hence,
\[
\int_{-2\sqrt{q}}^{2\sqrt{q}}\ee^{\ii t\lambda}\Psi'(\lambda)\Phi_n(\lambda)\,\dd\lambda = \begin{cases} \ii\Big(\frac{\pi}{\sqrt{q}\,t}\Big)^{1/2}\sin\left(2\sqrt{q}t-\frac{\pi}{4}\right)\cdot g(0)\Phi_n(2\sqrt{q}) +O(t^{-1})& \text{if } n \text{ is even},\\ \left(\frac{\pi}{\sqrt{q}\,t}\right)^{1/2}\cos\left(2\sqrt{q}t-\frac{\pi}{4}\right)\cdot g(0)\Phi_n(2\sqrt{q}) +O(t^{-1}) & \text{if } n \text{ is odd}.\end{cases}
\]
As $g(0) = \frac{-\sqrt{q}(q+1)}{(q-1)^2}$, $\Phi_n(2\sqrt{q}) = \frac{q^{-n/2}}{q+1}(2+(n+1)(q-1))$, $\cos(2\sqrt{q}t-\frac{\pi}{4}) = \sin(2\sqrt{q}t+\frac{\pi}{4})$ then recalling \eqref{e:genkerA}, \eqref{e:errorder2}, we obtain precisely the expression given in Theorem~\ref{thm:maincomb}.

The dispersive estimate \eqref{e:maindis} follows immediately since $\frac{q^{\frac{-n}{2}+\frac{1}{4}}(2+(n+1)(q-1))}{(q-1)^2}\To 0$ as $n\To\infty$, so it can be uniformly bounded by some $c_q$ for all $n$.

\section{The case of quantum trees}\label{sec:quancase}

We now move to quantum graphs; our aim is to prove Theorem~\ref{thm:mainquan}. The basic ideas are the same as the previous section~: 
\begin{itemize}
\item The Schwartz kernel of the evolution operator is essentially the Fourier transform of the spectral quantity $f_{x,y}(\lambda)= \Im G^{\lambda+\ii 0}(x,y)$.
\item Integrating once by parts gives a first decay of $t^{-1}$, multiplied by the Fourier transform of $\partial_{\lambda} f_{x,y}(\lambda)$. After a change of variable, the stationary phase lemma shows this last Fourier transform is $\asymp t^{-1/2}$. So in total we get the speed $t^{-3/2}$.
\end{itemize} 

The proof however takes a lot more effort here. The difficulty is not really in guessing the leading term, it is instead in controlling the error term $O(t^{-2})$ carefully. Like the previous section, having a rather explicit error is useful to get a control uniform in $x,y$. But in addition to this, because we have infinitely many spectral bands here, we get infinite series of error terms upon applying the stationary phase lemma to each spectral integral, so we should show these series of errors do converge.

\subsection{Preliminary spectral analysis} \label{sec:prelimspe}

Consider the solutions of $q\mu^2-w(z)\mu+1=0$~:

\[
\mu^\pm(z) = \frac{w(z)\pm \sqrt{w(z)^2-4q}}{2q}\,.
\]

As we will see, they arise in the computation of the Green's function of $H_{\mathbf{T}_q}$.

\begin{rem}\label{rem:rootchoice}
If $w(z)^2-4q = r\ee^{\ii \phi}$, we should explain if $\sqrt{w(z)^2-4q}$ denotes $\sqrt{r}\ee^{\ii\phi/2}$, the root with positive imaginary part, or $-\sqrt{r}\ee^{\ii\phi/2}$, the one with negative imaginary part.

Note that $q\mu^+(z)\mu^-(z)=1$. If $z\in \C^+$, then one solution has modulus $<\frac{1}{\sqrt{q}}$ while the other is $>\frac{1}{\sqrt{q}}$. As in \cite[Section 3]{Carl}, we always choose the branch of the square root such that $|\mu^-(z)|<\frac{1}{\sqrt{q}}$. For example, if $W=\alpha=0$, $q=1$ and $L=1$, we have $\mu^-(z) = \cos\sqrt{z} - \sqrt{-\sin^2(\sqrt{z})}$. Then with this convention $\mu^-(z)=\cos\sqrt{z}+\ii\sin\sqrt{z}=\ee^{\ii \sqrt{z}}$ is the correct choice, as $|\ee^{\ii \sqrt{z}}|<1$ for $z\in \C^+$.

In general, one has $|\mu^-(z)|^2<\frac{1}{q}\iff \Re w(z)\overline{\sqrt{w(z)^2-4q}}>0$. One easily checks that to have this inequality, when $w(z)\in \C^+$, we should choose the branch with positive imaginary part, and when $w(z)\in \C^-$ we choose the branch with negative imaginary part. This is coherent with the above special case, since the root $-\ii \sin\sqrt{z}$ has positive imaginary part iff $\cos\sqrt{z}$ has.

If $\lambda\in \R$, we denote $\mu^{\pm}(\lambda):=\lim_{\eta \downarrow 0}\mu^{\pm}(\lambda+\ii\eta)$. Note that if $\lambda\in \R$, then $w(\lambda)\in \R$. If moreover $|w(\lambda)|\le 2\sqrt{q}$, then $|\mu^{\pm}(\lambda)| = \frac{1}{\sqrt{q}}$.
\end{rem}

It can be shown \cite{Carl,ISW} that the resolvent kernel is given by
\begin{equation}\label{e:goo}
G^z_{\mathbf{T}_q}(o,o) = \frac{-s(z)}{(q+1)\mu^-(z)-w(z)}
\end{equation}
at any vertex $o$. If $x,y$ are in the same edge $e$, then
\begin{multline}\label{e:diffsame}
G^z_{\mathbf{T}_q}(x,y) = \frac{G^z_{\mathbf{T}_q}(o,o)}{s^2(z)}\big(S_z(L-x)S_z(L-y)+S_z(x)S_z(y)\\+\mu^-(z)S_z(L-x)S_z(y)+q\mu^+(z)S_z(x)S_z(L-y)\big)\,.
\end{multline}
If $x,y$ are in different edges a distance $n$ apart, then
\begin{equation}\label{e:diffdiffq}
G^z_{\mathbf{T}_q}(x,y) = \big(\mu^-(z)\big)^nG^z_{\mathbf{T}_q}(x,y')\,,
\end{equation}
where $y'$ is the point $y$ moved to the same edge as $x$, in the respective position.

All limits
\begin{equation}\label{e:grinlim}
G^\lambda_{\mathbf{T}_q}(x,y):=\lim_{\eta\downarrow 0}G^{\lambda+\ii\eta}_{\mathbf{T}_q}(x,y)
\end{equation}
exist if $\lambda\in \sigma_{ac}(H)$. Moreover $\lambda\mapsto G^{\lambda}(x,y)$ is analytic since all functions given above are analytic by classical ODE results.

The spectrum of $H$ is given as follows~: let
\[
\sigma_D = \{\lambda\in \R:s(\lambda)=0\}
\]
be the \emph{Dirichlet values}, i.e. $\sigma_D=\{(\frac{n\pi}{L})^2\}_{n\ge 1}$ in the special case $W=0$. Define
\[
\sigma_{ac} = \{\lambda\in \R : |w(\lambda)|\le 2\sqrt{q}\}\,.
\]

Then \cite{Carl} $\sigma_{ac}$ consists of AC spectrum and
\[
\sigma(H) = \sigma_{ac}\cup \sigma_D\,.
\]

Moreover $\sigma_{ac}\cap \sigma_D = \emptyset$ whenever $q\ge 2$. If we denote the $n$-th Dirichlet value by $\delta_n$, $n\ge 1$, we know that in each of the intervals $(-\infty,\delta_1)$, $(\delta_n,\delta_{n+1})$, for any $E\in [-2\sqrt{q},2\sqrt{q}]$, there is a unique $\lambda$ such that $w(\lambda)=E$. We also know that $\partial_{\lambda}w(\lambda)\neq 0$ on $\sigma_{ac}$, see \cite[Theorem 4.3]{Carl}.

Summarizing, the preceding considerations imply the following picture~:
\begin{equation}\label{e:acspec}
\sigma_{ac} = \cup_{n\ge 1} I_n\,,
\end{equation}
between the bands of AC spectrum $I_n$ and $I_{n+1}$ lies the eigenvalue $\delta_n$ (disjoint from them). Moreover, $w(I_n) = [-2\sqrt{q},2\sqrt{q}]$ for any $n$, and $w$ is either strictly increasing or strictly decreasing on $I_n$. To complete the picture, it is useful to remember that in general the $n$-th Dirichlet value $\delta_n$ is given by
\begin{equation}\label{e:dirval}
\delta_n = \Big(\frac{n\pi}{L}\Big)^2+O(1) \,,
\end{equation}
see \cite[Theorem 4, p.35]{PT87}. In the special case $W=0$, we have $I_n = [(\frac{(n-1)\pi+\theta}{L})^2,(\frac{n\pi-\theta}{L})^2]$, where $\theta=\arccos \frac{2\sqrt{q}}{q+1}$.

Let $\mathbf{1}_{ac}(H)$ be the orthogonal projection onto $\mathscr{H}_{ac}$. We now verify that $\ee^{\ii tH}\mathbf{1}_{ac}(H)$ is an integral operator with kernel given by
\begin{equation}\label{e:contker}
\ee^{\ii tH}\mathbf{1}_{ac}(H)(x,y)=\frac{1}{\pi}\int_{\sigma_{ac}(H)}\ee^{\ii t\lambda}\Im G^{\lambda}_{\mathbf{T}_q}(x,y)\,\dd\lambda\,.
\end{equation}
Recall that we denote $G^\lambda_{\mathbf{T}_q}(x,y):= \ds \lim_{\eta\downarrow 0}G^{\lambda+\ii\eta}_{\mathbf{T}_q}(x,y)$ for $\lambda\in\sigma_{ac}(H)$.

\begin{lem}
We have $(\ee^{\ii tH}\mathbf{1}_{ac}(H)\phi)(x)=\int_{\mathbf{T}_q}\left(\frac{1}{\pi}\int_{\sigma_{ac}}\ee^{\ii t\lambda}\Im G^{\lambda}_{\mathbf{T}_q}(x,y)\,\dd\lambda\right)\phi(y)\,\dd y$ for any $\phi$ of compact support in $\mathbf{T}_q$.
\end{lem}
\begin{proof}
We have by the spectral theorem
\[
\langle f,\ee^{\ii tH}\mathbf{1}_{ac}(H) g\rangle = \int_{\sigma_{ac}(H)}\ee^{\ii t\lambda}\,\dd\mu_{f,g}(\lambda)\,,
\]
where $\mu_{f,g}$ is the spectral measure at $f,g\in L^2(\mathbf{T}_q)$. We may assume $f,g\in \mathscr{H}_{ac}$, so $\mu_{f,g}$ is absolutely continuous. We claim the density is given by
\begin{equation}\label{e:muac}
\dd\mu_{f,g}(\lambda) = \frac{1}{\pi} \Im \langle f,G^{\lambda}_{\mathbf{T}_q}g\rangle\,\dd\lambda\,.
\end{equation}
In fact this holds for $f=g$ using \cite[Theorem 1.6(iv)]{SimSchro} and the fact that $\langle f,G^z_{\mathbf{T}_q}g\rangle = \int\frac{\dd\mu_{f,g}(x)}{x-z}$ is the Borel transform of $\mu_{f,g}$.
Using the identity
\begin{equation}\label{e:nani}
\langle f+g,A(f+g)\rangle -\langle f-g,A(f-g)\rangle = 2\langle f,A g\rangle + 2\langle g,A f\rangle
\end{equation}
with $A=G^\lambda$ and using that $G^\lambda(x,y)$ is symmetric in $x,y$ by construction, which implies that $\langle f,G^\lambda g\rangle=\langle g,G^\lambda f\rangle$ for real $f,g$, we get
\begin{align*}
\frac{1}{\pi}\int_I \Im \langle f,G^\lambda g\rangle\,\dd\lambda &= \frac{1}{\pi}\int_I\frac{\Im \langle f+g,G^\lambda (f+g)\rangle - \Im\langle f-g,G^\lambda (f-g)\rangle}{4}\,\dd\lambda \\
&= \frac{\mu_{f+g}(I)-\mu_{f-g}(I)}{4} = \frac{\mu_{f,g}(I)+\mu_{g,f}(I)}{2}\,,
\end{align*}
where $\mu_\phi:=\mu_{\phi,\phi}$ and we used \eqref{e:nani} with $A=\mathbf{1}_I(H)$ in the last equality. But the symmetry of $G^{z}(x,y)$ implies $\mu_{f,g}=\mu_{g,f}$, since $\frac{\mu_{f,g}[a,b]+\mu_{f,g}(a,b)}{2}=\lim_{\eta\downarrow 0}\frac{1}{\pi}\int_a^b\Im \langle f,G^{\lambda+\ii\eta}g\rangle\,\dd\lambda$ by Fubini's theorem. This completes the proof of \eqref{e:muac}. It follows that 
\[
\langle f,\ee^{\ii tH}\mathbf{1}_{ac}(H)g\rangle = \frac{1}{\pi}\int_{\sigma_{ac}(H)}\ee^{\ii t\lambda}\Im \langle f,G^\lambda_{\mathbf{T}_q} g\rangle \,\dd\lambda\,,
\]
so for real-valued $f,g$ of compact support,
\[
\langle f,\ee^{\ii tH}\mathbf{1}_{ac}(H)g\rangle =\frac{1}{\pi}\int_{\sigma_{ac}(H)}\int_{\mathbf{T}_q\times\mathbf{T}_q}\ee^{\ii t\lambda}f(x)g(y)\Im G^{\lambda}_{\mathbf{T}_q}(x,y)\,\dd\lambda\,\dd x\dd y\,.
\]
The same holds for complex-valued $f,g$ by expanding $f=f_1+\ii f_2$, $g=g_1+\ii g_2$, replacing $f(x)$ by $\overline{f(x)}$ in the end. This completes the proof.
\end{proof}

\begin{rem}\label{rem:kernel}
The lemma extends to any $f\in L^1(\mathbf{T}_q)\cap L^2(\mathbf{T}_q)$. In fact, if $F(H)=\ee^{\ii tH}\mathbf{1}_{ac}(H)$ and $(f_j)$ have compact support and tend to $f$ in $L^1\cap L^2$ (take e.g. $f_j = f\mathbf{1}_{\Lambda_j}$ for a sequence of growing cubes $\Lambda_j$), we get $|\int F(H)(x,y)f(y)\dd y-\int F(H)(x,y)f_j(y)\dd y|\le \|F(H)(\cdot,\cdot)\|_{\infty}\|f-f_j\|_1 \To 0$ provided the kernel is uniformly bounded. On the other hand $F(H)f_j\To F(H)f$ in $L^2$, so up to extracting a subsequence it also converges almost everywhere. Thus, $[F(H)f](x) = \lim_j [F(H)f_j](x) = \lim_j \int F(H)(x,y)f_j(y)\dd y = \int F(H)(x,y) f(y)\dd y$ as required. In particular, $\|F(H)f\|_{\infty}\le \|F(H)(\cdot,\cdot)\|_{\infty}\|f\|_1$ for any $f\in L^1(\mathbf{T}_q)\cap L^2(\mathbf{T}_q)$. We shall hence focus on the analysis of the kernel, the dispersive estimate will follow.
\end{rem}   

We now calculate $\Im G_{\mathbf{T}_q}^{\lambda}(o,o)$ for $\lambda\in \sigma_{ac}$. Note that both $s(\lambda),w(\lambda)\in \R$ for real $\lambda$. So if $z=\lambda+\ii\eta$, then taking $\eta\downarrow 0$ in \eqref{e:goo}, we have $\Im G^{\lambda}_{\mathbf{T}_q}(o,o)=\frac{(q+1)s(\lambda)\Im \mu^{-}(\lambda)}{|(q+1)\mu^-(\lambda)-w(\lambda)|^2}$.

Concerning $\mu^-(z)$, recall by Remark~\ref{rem:rootchoice} that the choice of the branch depends on $\Im w(z)$. For $\eta>0$ small enough, we have $w(\lambda+\ii\eta) = w(\lambda) +\ii\eta w'(\lambda) +o(\eta^2)$, implying that $\sgn \Im w(\lambda+\ii\eta) = \sgn w'(\lambda)$. Hence, we choose the positive branch of the square root if $\sgn w'(\lambda)=1$ and the negative branch otherwise.

For $\lambda\in \sigma_{ac}$, we have by definition $|w(\lambda)|\le 2\sqrt{q}$. Hence, the previous considerations imply that $\mu^-(\lambda) = \frac{w(\lambda)-\ii\eps_{\lambda}\sqrt{4q-w(\lambda)^2}}{2q}$, where $\eps_{\lambda}=\sgn w'(\lambda)$ and $\sqrt{4q-w(\lambda)^2}\in [0,2\sqrt{q})$ is the usual nonnegative square root. Moreover,
\begin{multline*}
|(q+1)\mu^-(\lambda)-w(\lambda)|^2 = \Big|(q+1)\frac{w(\lambda)-\ii\eps_{\lambda}\sqrt{4q-w(\lambda)^2}}{2q}-w(\lambda)\Big|^2 \\
= w(\lambda)^2\Big(\frac{q+1}{2q}-1\Big)^2+(q+1)^2\frac{4q-w(\lambda)^2}{(2q)^2}=\frac{(q+1)^2-w(\lambda)^2}{q}\,.
\end{multline*}
Hence,
\begin{equation}\label{e:imoo}
\Im G^{\lambda}_{\mathbf{T}_q}(o,o) = \frac{(q+1)s(\lambda)\Im \mu^{-}(\lambda)}{|(q+1)\mu^-(\lambda)-w(\lambda)|^2} = \frac{-s(\lambda)\eps_{\lambda}(q+1)\sqrt{4q-w(\lambda)^2}}{2[(q+1)^2-w(\lambda)^2]}\,.
\end{equation}

In other words, if $\Psi(\lambda) = \Im G^\lambda_{\cA}(o,o) = \frac{(q+1)\sqrt{4q-\lambda^2}}{2[(q+1)^2-\lambda^2]}$ as in Section~\ref{sec:comb}, and if we now put $\Psi_1(\lambda) = \Im G^\lambda_{\mathbf{T}_q}(o,o)$, we get
\[
\Psi_1(\lambda)=-\eps_{\lambda}s(\lambda)\Psi(w(\lambda))\,.
\]

To find the sign of $w'(\lambda)$, we note that $w$ has a kind of weak periodicity. More precisely, by \cite[Theorem 4.3]{Carl}, $w(\delta_n)=(-1)^n(q+1)$. In particular $w(\delta_{2n})=(q+1)$. Now $\lambda\mapsto w(\lambda)$ is analytic and bounded on $\R$, so $w'(\lambda)$ has a discrete set of zeroes and $w$ is strictly monotone in between. Since $\cup_{k\ge 1} I_k=\{\lambda:|w(\lambda)|\le 2\sqrt{q}\}$ and $I_{2n}=[a_{2n},b_{2n}]$ lies on the left of $\delta_{2n}$, we must have $w(b_{2n})=2\sqrt{q}$, so $w(a_{2n})=-2\sqrt{q}$. Similarly $w(a_{2n+1})=2\sqrt{q}$ and $w(b_{2n+1})=-2\sqrt{q}$. Hence, $|w'(\lambda)| = (-1)^nw'(\lambda)$ on $I_n$. Thus, for $\lambda\in I_n$,
\begin{equation}\label{e:psi1}
\Psi_1(\lambda)=(-1)^{n+1}s(\lambda)\Psi(w(\lambda))\,.
\end{equation}

As one may expect\footnote{The nonnegativity of $\Psi_1$ is a general fact~: for any quantum graph, $G^{z}(x,x)$ is a Herglotz function. While this is an immediate consequence of the spectral theorem in case of discrete graphs, for quantum graphs the statement is nontrivial; see \cite[Lemma A.8]{AISW2} for a proof.}, $\Psi_1(\lambda)\ge 0$ over $\sigma_{ac}$. In fact, $s(\delta_n)=0$ and $\partial_{\lambda}s(\delta_n)\neq 0$ by \cite[p.30]{PT87}, so $s(\lambda)$ changes sign when crossing $\delta_n$. It is shown in \cite[Section 4.3]{AISW} that $s(\lambda)>0$ for $\lambda<\delta_1$. In particular, $s(\lambda)>0$ on $I_1$, so $\sgn s(\lambda) = (-1)^{n+1}$ on $I_n$.

\subsection{Handling diagonal terms}\label{sec:diag}

Let us turn back to the kernel \eqref{e:contker}. Analogously to the discrete case, \eqref{e:diffsame} and \eqref{e:diffdiffq} suggest that we first prove Theorem~\ref{thm:mainquan} for $x=y=o$, then generalize later. Most of the technical difficulties already appear here. 

In view of \eqref{e:acspec} and the comments thereafter, we have $w(\lambda)=\pm 2\sqrt{q}$ at the endpoints of the bands $I_n$. Hence, $\Psi_1(\lambda) = 0$ at these endpoints. Integrating by parts thus yields
\begin{align*}
\ee^{\ii tH}\mathbf{1}_{ac}(H)(o,o) &= \frac{-1}{\pi \ii t}\int_{\sigma_{ac}(H)}\ee^{\ii t\lambda}\Psi_1'(\lambda)\,\dd\lambda\\
&=\frac{-\ii}{\pi t}\sum_{n=1}^{\infty}(-1)^n\int_{I_n}\ee^{\ii t \lambda}[s'(\lambda)\Psi(w(\lambda)) + s(\lambda)\Psi'(w(\lambda))w'(\lambda)]\,\dd\lambda\\
&=\frac{-\ii}{\pi t}\sum_{n=1}^{\infty}(-1)^n\int_{I_n}\ee^{\ii t \lambda}\Big[\frac{s'(\lambda)\Psi(w(\lambda))}{w'(\lambda)} + s(\lambda)\Psi'(w(\lambda))\Big]w'(\lambda)\,\dd\lambda
\end{align*}
using \eqref{e:psi1}. At this point it is natural to consider the change of variables $w(\lambda)=2\sqrt{q}\cos\theta$ (recall the discussion around \eqref{e:acspec}). We showed above that $|w'(\lambda)| = (-1)^nw'(\lambda)$ on $I_n$. Letting $w_n$ be the restriction of $w$ to $I_n$, it follows that
\begin{multline}\label{e:anothereq}
\ee^{\ii tH}\mathbf{1}_{ac}(H)(o,o) = \frac{-\ii}{\pi t}\sum_{n=1}^\infty\int_0^\pi\ee^{\ii t w_n^{-1}(2\sqrt{q}\cos\theta)}\\
\bigg[\frac{s'(w_n^{-1}(2\sqrt{q}\cos\theta))(q+1)4q\sin^2\theta}{2[(q+1)^2-4q\cos^2\theta]w'(w_n^{-1}(2\sqrt{q}\cos\theta))}
+ s(w_n^{-1}(2\sqrt{q}\cos\theta))g(\theta)\bigg]\,\dd\theta\,,
\end{multline}
with $g(\theta)$ given by \eqref{e:g}. The phase function is now $\phi_n(\theta)=w_n^{-1}(2\sqrt{q}\cos\theta)$. So $\phi'_n(\theta) = \frac{-2\sqrt{q}\sin\theta}{w'(w_n^{-1}(2\sqrt{q}\cos\theta))}$. This suggests the first integrand already decays faster than required. In fact, if $h_n(\theta)=\frac{s'(\phi_n(\theta))(q+1)2\sqrt{q}\sin\theta}{2[(q+1)^2-4q\cos^2\theta]}$, then integrating by parts, the first term gives
\begin{equation}\label{e:someeq}
-\int_0^\pi \{\ee^{\ii t\phi_n(\theta)}\phi'_n(\theta)\}h_n(\theta)\,\dd\theta
= -\frac{\ee^{\ii t\phi_n(\theta)}}{\ii t}h_n(\theta)\Big|_0^\pi + \frac{1}{\ii t}\int_0^\pi\ee^{\ii t\phi_n(\theta)}h'_n(\theta)\,\dd\theta\,.
\end{equation}
Now $h_n(0)=h_n(\pi)=0$. Moreover,
\begin{multline}\label{e:theactualone}
h_n'(\theta) = (q+1)\sqrt{q}\bigg\{\frac{s''(\phi_n(\theta))\frac{-2\sqrt{q}\sin\theta}{w'(w_n^{-1}(2\sqrt{q}\cos\theta))}\sin\theta +s'(\phi_n(\theta))\cos\theta}{[(q+1)^2-4q\cos^2\theta]} \\
- \frac{[s'(\phi_n(\theta))\sin\theta][8q\cos\theta\sin\theta]}{[(q+1)^2-4q\cos^2\theta]^2}\bigg\}\,.
\end{multline}

This quantity is bounded over $[0,\pi]$, so by the trivial modulus bound, \eqref{e:someeq} is $O(t^{-1})$, which amounts to a contribution of $O(t^{-2})$ in \eqref{e:anothereq}, indeed faster than $t^{-3/2}$. There is one point however that needs to be checked in the present context~: does the corresponding series of such ``error integrals'' converge? we need to study the decay in $n$. Such decay will be obtained from the $s',s''$ terms,\footnote{Note that here differentiation is w.r.t. energy. This is not the notation used in \cite{Carl}.} see \eqref{e:sdeca} and \eqref{e:probliquot}.

For transparency consider first the case $W=0$. Then $s(\lambda) = \frac{\sin\sqrt{\lambda}L}{\sqrt{\lambda}}$, so $s'(\lambda) = \frac{L\cos\sqrt{\lambda}L}{2\lambda}-\frac{\sin\sqrt{\lambda}L}{2\lambda^{3/2}}$ and $s''(\lambda) = \frac{-L^2\sin\sqrt{\lambda}L}{4\lambda^{3/2}}-\frac{3L\cos\sqrt{\lambda}L}{4\lambda^2}+\frac{3\sin\sqrt{\lambda}L}{4\lambda^{5/2}}$. We thus have $s'(\lambda)=O(\lambda^{-1})$ and $s''(\lambda)=O(\lambda^{-3/2})$. On the other hand, $\phi_n(\theta)\in I_n$ by definition, so we deduce from \eqref{e:dirval} that $s'(\phi_n(\theta))=O(n^{-2})$ and $s''(\phi_n(\theta))=O(n^{-3})$. This is indeed summable.

In general, as shown in \cite[Theorem 1, p.7]{PT87}, we have a series expansion in the form
\begin{equation}\label{e:sersin}
s(\lambda) = \frac{\sin\sqrt{\lambda}L}{\sqrt{\lambda}} + \sum_{k\ge 1}S_k(\lambda)\,,
\end{equation}
where, denoting $\Delta_k:=\{0\le t_1\le \dots\le t_{k+1}=L\}\subset \R^k$, we have
\[
S_k(\lambda) = \int_{\Delta_k} \frac{\sin\sqrt{\lambda} t_1\sin\sqrt{\lambda}(t_2-t_1)\cdots \sin\sqrt{\lambda}(t_{k+1}-t_k)}{\lambda^{(k+1)/2}}W(t_1)\cdots W(t_k)\,\dd t_1\cdots \dd t_k\,.
\]

Note that $\Delta_2$ is just a triangle of area $\frac{1}{2}L^2$. In general as explained in \cite{PT87}, $\Delta_k$ has volume $\frac{L^k}{k!}$. This is done by writing the cube $[0,L]^k$ as the union of $\Delta_k$ and its permutations of coordinates (e.g. $[0,L]^2 = \{0\le t_1\le t_2\le L\}\cup \{0\le t_2\le t_1\le L\}$). As each of these pieces has the same volume we get $L^k = k!\, |\Delta_k|$ as required.

It follows that $|S_k(\lambda)|\le \frac{\|W\|_{\infty}^k}{\lambda^{(k+1)/2}}\cdot \frac{L^k}{k!}$. For clarity, fix $t_j\in [0,L]$ and let
\[
\sigma_j(\lambda) = \frac{\sin\sqrt{\lambda}(t_{j+1}-t_j)}{\sqrt{\lambda}}\,,\qquad f_k(\lambda) = \prod_{j=0}^k\sigma_j(\lambda)\,,
\]
so that if $t_0:=0$, then $\ds S_k(\lambda)=\int_{\Delta_k}f_k(\lambda)  \prod_{j=1}^kW(t_j)\dd t_j$. Then
\[
f_k'(\lambda) = \sum_{j_1=0}^k \sigma_{j_1}'(\lambda)\prod_{\substack{j_2\le k,\\j_2\neq j_1}} \sigma_{j_2}(\lambda)\,,
\]
\begin{equation}\label{e:f''}
f_k''(\lambda) = \sum_{j_1=0}^k\sigma_{j_1}''(\lambda)\prod_{\substack{j_2\le k,\\j_2\neq j_1}} \sigma_{j_2}(\lambda) + \sum_{j_1=0}^k\sum_{\substack{j_2\le k,\\j_2\neq j_1}}\sigma_{j_1}'(\lambda)\sigma_{j_2}'(\lambda)\prod_{\substack{j_3\le k,\\j_3\neq j_1,j_2}} \sigma_{j_3}(\lambda)\,.
\end{equation}
Since $\sigma_j'(\lambda)=\frac{(t_{j+1}-t_j)\cos\sqrt{\lambda}(t_{j+1}-t_j)}{2\lambda}-\frac{\sin\sqrt{\lambda}(t_{j+1}-t_j)}{2\lambda^{3/2}}$, we get $|f_k'(\lambda)|\le \frac{L(k+1)}{\lambda^{(k+2)/2}}$ for any $t_j\in [0,L]$. Similarly, using $\sigma_j''(\lambda)=O(\lambda^{-3/2})$, we get $|f_k''(\lambda)|\le \frac{C_L(k+1)^2}{\lambda^{(k+3)/2}}$. Using dominated convergence, it follows that $|S_k'(\lambda)|\le \frac{L(k+1)}{\lambda^{(k+2)/2}}\cdot \frac{\|W\|_{\infty}^kL^k}{k!}$ and $|S_k''(\lambda)|\le \frac{C_L(k+1)^2}{\lambda^{(k+3)/2}}\cdot \frac{\|W\|_{\infty}^kL^k}{k!}$.

In general, for $\lambda\ge 1$ and $m\ge 0$,
\begin{equation}\label{e:dersumest}
\sum_{k\ge 1}\frac{(k+1)^m}{k!}\frac{\|W\|_\infty^k L^k}{\lambda^{\frac{k+m+1}{2}}}\le \frac{2^m\|W\|_{\infty}L}{\lambda^{\frac{m+2}{2}}}\sum_{k\ge 1}\frac{k^{m}(\|W\|_\infty L)^{k-1}}{k!} = O(\lambda^{-\frac{m+2}{2}})\,.
\end{equation}

Applying this for $m=0,1,2$, using dominated convergence, we conclude from \eqref{e:sersin} that

\begin{equation}\label{e:sdeca}
s(\lambda) = \frac{\sin\sqrt{\lambda}L}{\sqrt{\lambda}}+O(\lambda^{-1})\,,\qquad s'(\lambda) = \frac{L\cos\sqrt{\lambda}L}{2\lambda}+ O(\lambda^{-3/2})\,,
\end{equation}

\begin{equation}\label{e:s''}
s''(\lambda) = \frac{-L^2\sin\sqrt{\lambda}L}{4\lambda^{3/2}} + O(\lambda^{-2})\,.
\end{equation}

Turning back to \eqref{e:theactualone}, it remains to estimate $w'(\lambda)$. We already know from \cite[Theorem 4.3]{Carl} that $w'(\lambda)\neq 0$, but if it decays too fast with $\lambda$ it will be a problem.

Recalling \eqref{e:muw}, we have $w'(\lambda) = (q+1)c'(\lambda)+\alpha s'(\lambda)$. We already know that $s'(\lambda) = O(\lambda^{-1})$. On the other hand, there is also a series expansion for $c(\lambda)$ of the form
\begin{equation}\label{e:c}
c(\lambda) = \cos\sqrt{\lambda}L + \sum_{k\ge 1}\int_{\Delta_k}g_k(\lambda)\prod_{j=1}^kW(t_j)\,\dd t_j\,,
\end{equation}
with
\[
g_k(\lambda) = \frac{\cos\sqrt{\lambda}t_1\sin\sqrt{\lambda}(t_2-t_1)\cdots \sin\sqrt{\lambda}(t_{k+1}-t_k)}{\lambda^{k/2}}\,.
\]
Arguing as before, we find that $|g_k'(\lambda)|\le \frac{c_L(k+1)}{\lambda^{(k+1)/2}}$ and deduce that
\begin{equation}\label{e:c'}
c'(\lambda) = \frac{-L\sin\sqrt{\lambda}L}{2\sqrt{\lambda}}+ O(\lambda^{-1})\,.
\end{equation}

Recalling \eqref{e:sdeca}, we thus showed that
\begin{equation}\label{e:w'}
w'(\lambda) = \frac{-L(q+1)\sin\sqrt{\lambda}L}{2\sqrt{\lambda}}+O(\lambda^{-1})\,.
\end{equation}

Our last worry is the possibility that $\sin\sqrt{\lambda}L$ vanishes. Fortunately this never happens on $\sigma_{ac}$, at least if $\lambda$ is large. In fact, we know that $|w(\lambda)|\le 2\sqrt{q}$. As $w(\lambda)=(q+1)c(\lambda)+\alpha s(\lambda)$, and since $s(\lambda) = O(\lambda^{-1/2})$ by \eqref{e:sdeca}, we get $|(q+1)c(\lambda)|\le 2\sqrt{q}+ O(\lambda^{-1/2})$. By \cite[p.13]{PT87}, we deduce that $|(q+1)\cos\sqrt{\lambda}L|\le 2\sqrt{q}+O(\lambda^{-1/2}) + \frac{\ee^{\|W\|_{\infty}L}}{\sqrt{\lambda}}$. Consequently $\sin^2(\sqrt{\lambda}L) = 1-\cos^2(\sqrt{\lambda}L) \ge 1-\frac{4q}{(q+1)^2}-O(\lambda^{-1/2}) = \frac{(q-1)^2}{(q+1)^2}-O(\lambda^{-1/2})$. Assuming $\lambda$ is large, this is $\ge \frac{(q-1)^2}{4(q+1)^2}$. Thus,
\begin{equation}\label{e:w'min}
|w'(\lambda)|\ge \frac{L(q-1)}{4\sqrt{\lambda}}+O(\lambda^{-1})\,.
\end{equation}

Recalling \eqref{e:s''}, we finally obtain for large $\lambda$,
\begin{equation}\label{e:probliquot}
\frac{s''(\lambda)}{w'(\lambda)} = O(\lambda^{-1})\,.
\end{equation}

We may now turn back to \eqref{e:someeq}. Combining \eqref{e:sdeca}, \eqref{e:probliquot} and \eqref{e:dirval}, recalling that $w_n^{-1}(s)\in I_n$, and each $I_m$, $I_{m+1}$ are separated by $\delta_m$, we get
\begin{equation}\label{e:1stinteg}
-\int_0^\pi \{\ee^{\ii t\phi_n(\theta)}\phi'_n(\theta)\}h_n(\theta)\,\dd\theta = t^{-1}O(n^{-2})\,.
\end{equation}
The first integral on the RHS of \eqref{e:anothereq} is thus $t^{-2}O(n^{-2})$, so its series gives $t^{-2}O(1)$. 

\medskip

\textbf{Recap.} So far we have controlled the first integral on the RHS of \eqref{e:anothereq}. The rest of this subsection is devoted to the analysis of the second integral. Using the technique of stationary phase, we will see it is responsible for the main term in the asymptotics. Most of the effort will be to ensure that the series of error terms converges.

Recall that $\phi_n(\theta)=w_n^{-1}(2\sqrt{q}\cos\theta)$ and $\phi'_n(\theta) = \frac{-2\sqrt{q}\sin\theta}{w'(w_n^{-1}(2\sqrt{q}\cos\theta))}$. So
\begin{align}\label{e:phi''}
\phi_n''(\theta)&=\frac{-2\sqrt{q}\cos\theta}{w'(w_n^{-1}(2\sqrt{q}\cos\theta))}+\frac{2\sqrt{q}\sin\theta \cdot w''(w_n^{-1}(2\sqrt{q}\cos\theta))}{[w'(w_n^{-1}(2\sqrt{q}\cos\theta))]^2}\cdot \frac{-2\sqrt{q}\sin\theta}{w'(w_n^{-1}(2\sqrt{q}\cos\theta))}\nonumber\\
&=-\bigg[\frac{2\sqrt{q}\cos\theta}{w'(\phi_n(\theta))}+\frac{4q\sin^2\theta \cdot w''(\phi_n(\theta))}{[w'(\phi_n(\theta))]^3}\bigg]\,.
\end{align}
In $[0,\pi]$, $\phi_n$ has only critical points at $0,\pi$. If $I_n=[a_n,b_n]$, we have $\phi_n(0)\in \{a_n,b_n\}$, $\phi_n(\pi)\in \{a_n,b_n\}$, $\phi_n''(0) = \frac{-2\sqrt{q}}{w'(\phi_n(0))}$, $\phi_n''(\pi) = \frac{2\sqrt{q}}{w'(\phi_n(\pi))}$. If $\phi_n(0)=a_n$ then $w_n^{-1}(2\sqrt{q})=a_n$, so $w(a_n) = 2\sqrt{q}$ and thus $w(b_n)=-2\sqrt{q}$. Hence, $w$ decreases on $I_n$, i.e. $w'<0$, so $\phi_n''(0)=\frac{-2\sqrt{q}}{w'(a_n)}$ has positive sign while $\phi_n''(\pi)=\frac{2\sqrt{q}}{w'(b_n)}$ has negative sign. Similarly, if $\phi_n(0)=b_n$, then $w$ increases on $I_n$, so $\phi_n''(0)=\frac{-2\sqrt{q}}{w'(b_n)}$ has negative sign while $\phi_n''(\pi)=\frac{2\sqrt{q}}{w'(a_n)}$ has positive sign.

Applying Corollary~\ref{cor:phase}, \eqref{e:end}, we deduce that when $\phi_n(0)=a_n$, we have
\begin{multline}\label{e:phasresult}
\int_{0}^{\pi}\ee^{\ii t\phi_n(\theta)}s(\phi_n(\theta))g(\theta)\,\dd\theta \\
= \Big(\frac{\pi}{\sqrt{q}\,t}\Big)^{1/2}\cdot \frac{\ee^{\frac{\ii \pi}{4}}\ee^{\ii a_n t}|w'(a_n)|^{1/2}s(a_n)g(0) + \ee^{\frac{-\ii \pi}{4}}\ee^{\ii b_n t}|w'(b_n)|^{1/2}s(b_n)g(\pi)}{2} + O(t^{-1}).
\end{multline}
Recall that $g(\pi)=-g(0)$. Using this, we see that the expression when $\phi_n(0)=b_n$ is exactly the same, except for a sign change.

We showed prior to \eqref{e:psi1} that $w(a_{2n-1})=2\sqrt{q}$ and $w(b_{2n})=2\sqrt{q}$. It follows that $\phi_{2n-1}(0)=a_{2n-1}$ and $\phi_{2n}(0)=b_{2n}$. So in general \eqref{e:phasresult} should be multiplied by $(-1)^{n+1}$. In view of \eqref{e:anothereq} and \eqref{e:1stinteg}, we would like to conclude that
\begin{multline}\label{e:mainclaimdiag}
\ee^{\ii tH}\mathbf{1}_{ac}(H)(o,o)\\
=\frac{\ii g(0)}{\sqrt{\pi\sqrt{q}}\,t^{3/2}}\sum_{n\ge 1}(-1)^n \frac{\ee^{\frac{\ii \pi}{4}}\ee^{\ii a_n t}|w'(a_n)|^{1/2}s(a_n) - \ee^{\frac{-\ii \pi}{4}}\ee^{\ii b_n t}|w'(b_n)|^{1/2}s(b_n)}{2} +O(t^{-2}).
\end{multline}

In doing so however, we assert that the sum over $n$ of the terms $O(t^{-1})$ from \eqref{e:phasresult} converge. This is what we prove now, and we shall need the precise error estimate provided by Corollary~\ref{cor:phase} and \eqref{e:corprec}. We focus on the error terms at $0$, the same holds at $\pi$.

\medskip

We have $\frac{\dd}{\dd\theta}s(\phi_n(\theta))g(\theta)|_{\theta=0} =[s'(\phi_n(\theta))\phi_n'(\theta)g(\theta)+s(\phi_n(\theta))g'(\theta)]|_{\theta=0}$. But $\phi_n'(0)=0$ and $g'(0)=0$ as we see from \eqref{e:g}. So $\frac{\dd}{\dd\theta}s(\phi_n(\theta))g(\theta)|_{\theta=0}=0$. Next, $\phi_n'''(0)=0$, as we see from \eqref{e:phi''}. By \eqref{e:q11a}, this shows that $Q_{1,1}(0)=0$ (as in the combinatorial case). 

Next, assuming $\phi_n(0)=a_n$, we have
\[
Q_{1,1}(\theta) = \frac{s(\phi_n(\theta))g(\theta)w'(\phi_n(\theta))}{-2\sqrt{q}\sin\theta}-\frac{|w'(a_n)|^{1/2}s(a_n)g(0)}{2q^{1/4}\sqrt{\phi_n(\theta)-a_n}}\,.
\]
In particular $|Q_{1,1}(\frac{\pi}{2})| = \frac{|s(a_n)g(0)||w'(a_n)|^{1/2}}{2q^{1/4}\sqrt{\phi_n(\frac{\pi}{2})-a_n}}$. Recalling that $a_n \asymp n^2$, we know from \eqref{e:sdeca} and \eqref{e:w'} that $s(a_n)|w'(a_n)|^{1/2} \asymp n^{-3/2}$. Next, as $\phi_n$ is monotone,
\[
\sqrt{\phi_n\Big(\frac{\pi}{2}\Big)-a_n}\ge \sqrt{\phi_n\Big(\frac{\pi}{2}\Big)-\phi_n\Big(\frac{\pi}{4}\Big)} \ge \frac{\sqrt{\pi}}{2}  \inf_{\theta\in [\frac{\pi}{4},\frac{\pi}{2}]}\sqrt{\phi_n'(\theta)} \,.
\]
Now $\phi_n'(\theta)=\frac{-2\sqrt{q}\sin\theta}{w'(\phi_n(\theta))}$. Over $[\frac{\pi}{4},\frac{\pi}{2}]$, $\sin\theta\ge \frac{1}{\sqrt{2}}$, while by \eqref{e:w'}, we know $|w'(\lambda)|\lesssim \frac{L(q+1)}{\sqrt{\lambda}}$. We conclude that $\frac{1}{\sqrt{\phi_n(\frac{\pi}{2})-a_n}}\lesssim \frac{C}{\sqrt{n}}$. So the term $|Q_{1,1}(\frac{\pi}{2})|$ from Corollary~\ref{cor:phase} is $O(n^{-2})$, which is summable in $n$, and note that we also controlled the last error term in that corollary, which is just $2|Q_{1,1}(\frac{\pi}{2})|$ in our setting. So it only remains to control the total variation $V_{0,\frac{\pi}{2}}(Q_{1,1})$.

\medskip

We start by writing $Q_{1,1}$ in a less singular form. Using Taylor-Lagrange we expand
\[
\phi_n(\theta)-a_n = \frac{\phi_n''(0)}{2}\theta^2+\frac{\phi_n'''(\theta_0)}{6}\theta^3
\]
for some $\theta_0\in (0,\theta)$. Recalling $\phi_n''(0)=\frac{-2\sqrt{q}}{w'(a_n)}=\frac{2\sqrt{q}}{|w'(a_n)|}$, we have
\begin{align*}
Q_{1,1}(\theta)&=\frac{-s(\phi_n(\theta))w'(\phi_n(\theta))g(\theta)}{2\sqrt{q}\sin\theta}-\frac{|w'(a_n)|s(a_n)g(0)}{2\sqrt{q}\,\theta\sqrt{1+\frac{\phi_n'''(\theta_0)}{3\phi_n''(0)}\theta}} \\
&= \frac{F(0)-F(\theta)}{2\sqrt{q}\,\theta\sqrt{1+\frac{\phi_n'''(\theta_0)}{3\phi_n''(0)}\theta}} + \frac{w'(\phi_n(\theta))s(\phi_n(\theta))g(\theta)(\sin\theta-\theta)}{2\sqrt{q}\,\theta\sin\theta}\,,
\end{align*}
where
\[
F(\theta)=w'(\phi_n(\theta))s(\phi_n(\theta))g(\theta)\sqrt{1+\frac{\phi_n'''(\theta_0)}{3\phi_n''(0)}\theta}\,.
\]
Note that $\frac{\sin\theta-\theta}{\theta \sin\theta}\To 0$ by expanding $\sin\theta = \theta + \frac{\theta^3}{6}\cos\theta_1$ for some $\theta_1\in (0,\theta)$. Also, $\frac{F(\theta)-F(0)}{\theta}\To F'(0)$. To find $F'$, we observe that $1+\frac{\phi_n'''(\theta_0)}{3\phi_n''(0)}\theta = \frac{2(\phi_n(\theta)-a_n)}{\phi_n''(0)\theta^2}$. But
\begin{align}\label{e:derroot}
\frac{\dd}{\dd\theta}\sqrt{\frac{2(\phi_n(\theta)-a_n)}{\phi_n''(0)\theta^2}} &= \frac{1}{2\sqrt{\frac{2(\phi_n(\theta)-a_n)}{\phi_n''(0)\theta^2}}}\Big(\frac{2\phi_n'(\theta)}{\phi_n''(0)\theta^2}-\frac{4(\phi_n(\theta)-a_n)}{\phi_n''(0)\theta^3}\Big)\\
&=\frac{1}{\sqrt{1+\frac{\phi_n'''(\theta_0)}{3\phi_n''(0)}\theta}}\Big(\frac{\phi_n'(\theta)}{\phi_n''(0)\theta^2}-\frac{1}{\theta}\Big[1+\frac{\phi_n'''(\theta_0)}{3\phi_n''(0)}\theta\Big]\Big)\nonumber \\
&=\frac{1}{\phi_n''(0)\sqrt{1+\frac{\phi_n'''(\theta_0)}{3\phi_n''(0)}\theta}}\Big(\frac{\phi_n'''(\theta_2)}{2}-\frac{\phi_n'''(\theta_0)}{3}\Big)\,,\nonumber
\end{align}
where in the last step we expanded $\phi_n'(\theta)=\theta\phi_n''(0)+\frac{\theta^2}{2}\phi_n'''(\theta_2)$ for some $\theta_2\in(0,\theta)$. We may deduce that $F'(0)=0$ (using $\phi_n'(0)=g'(0)=\phi_n'''(0)=0$), confirming that $Q_{1,1}(0)=0$ as we saw before. More importantly, we have
\begin{align*}
Q_{1,1}'(\theta) &= \frac{-F'(\theta)\theta\sqrt{1+\frac{\phi_n'''(\theta_0)}{3\phi_n''(0)}\theta}+(F(\theta)-F(0))\big[\sqrt{1+\frac{\phi_n'''(\theta_0)}{3\phi_n''(0)}\theta}+\theta\frac{\dd}{\dd\theta}\sqrt{\frac{2(\phi_n(\theta)-a_n)}{\phi_n''(0)\theta^2}}\big]}{2\sqrt{q}\theta^2(1+\frac{\phi_n'''(\theta_0)}{3\phi_n''(0)}\theta)}\\
&\quad +\frac{\theta\sin\theta S'(\theta)-S(\theta)[\sin\theta+\theta\cos\theta]}{2\sqrt{q}\theta^2\sin^2\theta}\,,
\end{align*}
where $S(\theta)=w'(\phi_n(\theta))s(\phi_n(\theta))g(\theta)(\sin\theta-\theta)$. We may simplify this as
\begin{multline}\label{e:q11'2}
Q_{1,1}'(\theta) = \frac{F(\theta)-F(0)-F'(\theta)\theta}{2\sqrt{q}\theta^2\sqrt{1+\frac{\phi_n'''(\theta_0)}{3\phi_n''(0)}\theta)}} + \frac{(F(\theta)-F(0))\frac{\dd}{\dd\theta}\sqrt{\frac{2(\phi_n(\theta)-a_n)}{\phi_n''(0)\theta^2}}}{2\sqrt{q}\theta(1+\frac{\phi_n'''(\theta_0)}{3\phi_n''(0)}\theta)}\\
\quad +\frac{(w''sg+w's'g)\phi_n'(\theta)+w'sg'}{2\sqrt{q}}\cdot\frac{(\sin\theta-\theta)}{\theta\sin\theta} - \frac{w'sg}{2\sqrt{q}}\cdot\frac{(\sin^2\theta-\theta^2\cos\theta)}{\theta^2\sin^2\theta}\,,
\end{multline}
where we used the shorthand notation $w=w(\phi_n(\theta))$, $s=s(\phi_n(\theta))$, $g=g(\theta)$. Here, $\frac{F(\theta)-F(0)-F'(\theta)\theta}{\theta^2}\To \frac{-F''(0)}{2}$, $\frac{F(\theta)-F(0)}{\theta}\To F'(0)$, $\frac{\sin\theta-\theta}{\theta\sin\theta}\To 0$ as previously established and $\frac{\sin^2\theta-\theta^2\cos\theta}{\theta^2\sin^2\theta}\To \frac{1}{4}+\frac{1}{3}=\frac{7}{12}$ as easily checked.

To find $F''$, we note that by \eqref{e:derroot}, we have
\begin{multline}\label{e:d2sq}
\frac{\dd^2}{\dd\theta^2}\sqrt{\frac{2(\phi_n(\theta)-a_n)}{\phi_n''(0)\theta^2}} = \Big(\frac{\dd}{\dd \theta}\frac{1}{\sqrt{\frac{2(\phi_n(\theta)-a_n)}{\phi_n''(0)\theta^2}}}\Big)\Big(\frac{\phi_n'(\theta)}{\phi_n''(0)\theta^2}-\frac{2(\phi_n(\theta)-a_n)}{\phi_n''(0)\theta^3}\Big)\\
 + \frac{1}{\phi_n''(0)\sqrt{\frac{2(\phi_n(\theta)-a_n)}{\phi_n''(0)\theta^2}}}\left(\frac{\phi_n''(\theta)}{\theta^2}-\frac{4\phi_n'(\theta)}{\theta^3}+\frac{6(\phi_n(\theta)-a_n)}{\theta^4}\right).
\end{multline}
The first term is $\frac{1}{(...)^{3/2}}\left(\frac{\phi_n'''(\theta_2)}{2\phi_n''(0)}-\frac{\phi_n'''(\theta_0)}{3\phi_n''(0)}\right)^2$ by \eqref{e:derroot}. For the second, expand
\begin{align*}
\phi_n(\theta)-a_n &= \frac{\theta^2\phi_n''(0)}{2}+\frac{\theta^3\phi_n'''(0)}{6}+\frac{\theta^4\phi_n^{(4)}(\theta_a)}{24}\,,\\
\phi_n'(\theta)&=\theta\phi_n''(0)+\frac{\theta^2}{2}\phi_n'''(0)+\frac{\theta^3}{6}\phi_n^{(4)}(\theta_b)\,,\\
\phi_n''(\theta)&=\phi_n''(0)+\theta\phi_n'''(0)+\frac{\theta^2}{2}\phi_n^{(4)}(\theta_c)\,.
\end{align*}
Then the last bracket becomes $\big(\frac{\phi_n^{(4)}(\theta_a)}{4}-\frac{2\phi_n^{(4)}(\theta_b)}{3}+\frac{\phi_n^{(4)}(\theta_c)}{2}\big)$. As $\phi_n'''(0)=0$, this shows $\frac{\dd^2}{\dd\theta^2}\sqrt{\frac{2(\phi_n(\theta)-a_n)}{\phi_n''(0)\theta^2}} = \frac{\phi^{(4)}(0)}{12\phi''(0)}$ at zero, so that $F''(0)$ is finite.

Back to \eqref{e:q11'2}, let us estimate $\|Q_{1,1}'\|_{L^\infty[0,\frac{\pi}{2}]}$. We begin by controlling the last terms. Note that $g$ and its derivatives remain bounded. By \eqref{e:sdeca} and \eqref{e:w'}, $w's= O(\lambda^{-1})$. Since $a_n \asymp n^{-2}$, we see the sup norm of the last term is $O(n^{-2})$. This also controls the term $w'sg'$ by $O(n^{-2})$. Next, $\phi_n'(\theta)=O(\sqrt{\lambda})$. By \eqref{e:sdeca} and \eqref{e:w'}, we get $w's'g\phi_n(\theta) = O(\lambda^{-1})=O(n^{-2})$ in sup norm. To control $w''$, we first note that by \eqref{e:c} and \eqref{e:c'}, we have using the same arguments leading to \eqref{e:s''} that
\begin{equation}\label{e:c''}
c''(\lambda) = \frac{-L^2\cos\sqrt{\lambda}L}{4\lambda}+O(\lambda^{-3/2})\,.
\end{equation}
Recalling that $w=(q+1)c+\alpha s$, we see from \eqref{e:s''} that
\begin{equation}\label{e:w''}
w''(\lambda) = O(\lambda^{-1})\,.
\end{equation}
Thus, $w''sg\phi'(\theta) = O(\lambda^{-1})=O(n^{-2})$ is also summable in sup norm.

It remains to control the first two terms in \eqref{e:q11'2}, which we rewrite as
\begin{equation}\label{e:reducedF}
\frac{\frac{F''(\theta_1)}{2}-F''(\theta_2)}{2\sqrt{q}\sqrt{1+\frac{\phi_n'''(\theta_0)}{3\phi_n''(0)}\theta)}} + \frac{F'(\theta_0)\frac{\dd}{\dd\theta}\sqrt{\frac{2(\phi_n(\theta)-a_n)}{\phi_n''(0)\theta^2}}}{2\sqrt{q}(1+\frac{\phi_n'''(\theta_0)}{3\phi_n''(0)}\theta)}
\end{equation}
using Taylor-Lagrange, for some $\theta_j\in (0,\theta)$. We first show the denominator is well-behaved over $[0,\frac{\pi}{2}]$. Since $\phi_n''(0)=\frac{-2\sqrt{q}}{w'(a_n)}$, using \eqref{e:phi''}, one sees that
\begin{multline}\label{e:somequot}
\frac{\phi'''_n(\theta)}{\phi''_n(0)}= \frac{-w'(a_n)\sin\theta}{w'(\phi_n(\theta))}+\frac{3\sqrt{q}\sin2\theta w'(a_n)w''(\phi_n(\theta))}{w'(\phi_n(\theta))^3} -\frac{4q\sin^3\theta w'''(\phi_n(\theta))w'(a_n)}{w'(\phi_n(\theta))^4}\\
+\frac{12q\sin^3\theta w''(\phi_n(\theta))^2w'(a_n)}{w'(\phi_n(\theta))^5}
\end{multline}
say if $\phi_n(0)=a_n$. We claim the above stays bounded over $[0,\frac{\pi}{2}]$. In fact, $\phi_n(\theta)\in I_n=[a_n,b_n]$ and  $a_n,b_n$ are both $\asymp n^2$, so we see by \eqref{e:w'} and \eqref{e:w'min} that the first term in the RHS stays bounded. In view of \eqref{e:w''}, the second term is $O(n^{-1})O(n^{-2})O(n^3)=O(1)$. Similarly, the last term stays bounded. Finally we should control $w'''$. Differentiating \eqref{e:f''} and noting that $\sigma_j'''(\lambda)=O(\lambda^{-2})$, we see that $|f_k'''(\lambda)|\le \frac{c_L(k+1)^3}{\lambda^{(k+4)/2}}$ for large $\lambda$. Using \eqref{e:dersumest} with $m=3$, it follows that
\begin{equation}\label{e:s'''}
s'''(\lambda) = \frac{-L^3\cos\sqrt{\lambda}L}{8\lambda^2}+O(\lambda^{-5/2})\,.
\end{equation}
The same proof, using \eqref{e:c}, \eqref{e:c''} shows that
\begin{equation}\label{e:c'''}
c'''(\lambda) = \frac{L^3\sin\sqrt{\lambda}L}{8\lambda^{3/2}}+O(\lambda^{-2})\,.
\end{equation}
Hence,
\[
w'''(\lambda) = O(\lambda^{-3/2})\,.
\]
Back to \eqref{e:somequot}, we see the third term on the RHS is $O(n^{-3})O(n^{-1})O(n^4)=O(1)$. This completes the proof that $\frac{\phi'''_n(\theta)}{\phi_n''(0)}$ stays bounded over $[0,\frac{\pi}{2}]$, say $\left|\frac{\phi'''_n(\theta)}{\phi_n''(0)}\right|\le M$. 

For $\theta\in [0,\frac{1}{M}]$, we thus have $1+\frac{\phi'''_n(\theta_0)}{3\phi_n''(0)}\theta \ge \frac{2}{3}$. For $\theta\in [\frac{1}{M},\frac{\pi}{2}]$, as $\phi$ is increasing, we have
\[
1+\frac{\phi'''_n(\theta_0)}{3\phi_n''(0)} = \frac{2(\phi_n(\theta)-a_n)}{\phi_n''(0)\theta^2} \ge \frac{2(\phi_n(\frac{1}{M})-\phi_n(\frac{1}{2M}))}{\phi_n''(0)\theta^2}\ge \frac{4}{\pi^2M}\cdot \inf\limits_{\theta\in [\frac{1}{2M},\frac{1}{M}]} \frac{\phi_n'(\theta)}{\phi_n''(0)}\,.
\]
Recalling $\frac{\phi_n'(\theta)}{\phi_n''(0)} = \frac{w'(a_n)\sin\theta}{w'(\phi_n(\theta))}$, we get using \eqref{e:w'} and \eqref{e:w'min} that $1+\frac{\phi_n'''(\theta_0)}{3\phi_n''(0)}\theta\ge C_M>0$ over $[\frac{1}{M},\frac{\pi}{2}]$, with $C_M = \frac{(q-1)}{8\pi^2 M^2(q+1)}$ if $n\ge n_0=n_0(L,\|W\|_\infty)$.

Back to \eqref{e:reducedF}, using \eqref{e:derroot}, the second term takes the form $\frac{F'(\theta_0)(\frac{\phi_n'''(\theta_2)}{2\phi_n''(0)}-\frac{\phi_n'''(\theta_0)}{3\phi_n''(0)})}{(1+\frac{\phi_n'''(\theta_0)}{3\phi_n''(0)}\theta)^{3/2}}$. In view of the above arguments, this may be bounded by $C \|F'\|_\infty$ for some $C$ independent of $n$. But $F=w'sg\sqrt{1+\frac{\phi_n'''(\theta_0)}{3\phi_n''(0)}\theta}$, where $w'sg=w'(\phi_n(\theta))s(\phi_n(\theta))g(\theta)$ for short, so
\begin{equation}\label{e:F'}
F' = \bigg[(w''sg + w's'g)\phi_n'(\theta)+w'sg' + w'sg\frac{\frac{\phi_n'''(\theta_2)}{2\phi_n''(0)}-\frac{\phi_n'''(\theta_0)}{3\phi_n''(0)}}{1+\frac{\phi_n'''(\theta_0)}{3\phi_n''(0)}\theta}\bigg]\sqrt{1+\frac{\phi_n'''(\theta_0)}{3\phi_n''(0)}}\,.
\end{equation}
From the previous arguments, each of these terms is $O(n^{-2})$, so $\|F'\|_{L^\infty[0,\frac{\pi}{2}]}= O(n^{-2})$.

It only remains to bound the first term in \eqref{e:reducedF}. This amounts to control $\|F''\|_{L^\infty[0,\frac{\pi}{2}]}$. We have by \eqref{e:derroot},
\begin{multline}\label{e:F''}
F'' = \left[(w'''sg + 2w''s'g + w's''g)\phi_n'(\theta)^2+2[w''sg'+w's'g']\phi_n'(\theta)+w'sg''\right] \\
\times \sqrt{1+\frac{\phi_n'''(\theta_0)}{3\phi_n''(0)}} + 2\left\{(w''sg + w's'g)\phi_n'(\theta)+w'sg'\right\}\frac{\frac{\phi_n'''(\theta_2)}{2\phi_n''(0)}-\frac{\phi_n'''(\theta_0)}{3\phi_n''(0)}}{\sqrt{1+\frac{\phi_n'''(\theta_0)}{3\phi_n''(0)}\theta}}  \\
+ w'sg\frac{\dd^2}{\dd\theta^2}\sqrt{1+\frac{\phi_n'''(\theta_0)}{3\phi_n''(0)}} \,.
\end{multline}
The previous estimates over $w',s$ and their derivatives allow to conclude that each term above is $O(n^{-2})$, except perhaps $w'sg\frac{\dd^2}{\dd\theta^2}\sqrt{1+\frac{\phi_n'''(\theta_0)}{3\phi_n''(0)}}$. Since we know that $w'sg = O(n^{-2})$, it suffices to show that $\frac{\dd^2}{\dd\theta^2}\sqrt{1+\frac{\phi_n'''(\theta_0)}{3\phi_n''(0)}}$ stays bounded. We showed after \eqref{e:d2sq} that
\[
\frac{\dd^2}{\dd\theta^2}\sqrt{1+\frac{\phi_n'''(\theta_0)}{3\phi_n''(0)}} = \frac{(\frac{\phi_n'''(\theta_2)}{2\phi_n''(0)}-\frac{\phi_n'''(\theta_0)}{3\phi_n''(0)})^2}{(1+\frac{\phi_n'''(\theta_0)}{3\phi_n''(0)}\theta)^{3/2}} + \frac{\frac{\phi_n^{(4)}(\theta_a)}{4\phi_n''(0)}-\frac{2\phi_n^{(4)}(\theta_b)}{3\phi_n''(0)}+\frac{\phi_n^{(4)}(\theta_c)}{2\phi_n''(0)}}{\sqrt{1+\frac{\phi_n'''(\theta_0)}{3\phi_n''(0)}\theta}}\,.
\]

We have already controlled all these terms uniformly, except $\frac{\phi_n^{(4)}(\theta)}{\phi_n''(0)}$.

Recalling \eqref{e:somequot}, one easily checks that
\[
\Big|\frac{\phi_n^{(4)}(\theta)}{\phi_n''(0)}\Big|\le C_q|w'(a_n)|\max_{\substack{j\le 4\\ k\le 3}} \bigg(\Big|\frac{w^{(j)}(\phi_n(\theta))}{w'(\phi_n(\theta))^{j+1}}\Big|,\Big|\frac{w''(\phi_n(\theta))^k}{w'(\phi_n(\theta))^{2k+1}}\Big|,\Big|\frac{w''(\phi_n(\theta)w'''(\phi_n(\theta))}{w'(\phi_n(\theta))^6}\Big|\bigg).
\]

Our estimates imply that each term is $O(n^{-1})O(n)=O(1)$, except perhaps for the one containing $w^{(4)}$, which we now consider. It suffices to show that $w^{(4)}(\phi_n(\theta)) = O(n^{-4})$.

Differentiating \eqref{e:f''} twice using $\sigma^{(4)}_j(\lambda)=O(\lambda^{-5/2})$, we deduce that $|f_k^{(4)}(\lambda)|\le \frac{C(k+1)^4}{\lambda^{(k+5)/2}}$. Hence, using \eqref{e:dersumest} we get $s^{(4)}(\lambda)= O(\lambda^{-5/2})$. The same arguments show that $|g_k^{(4)}(\lambda)|\le \frac{\tilde{c}_L(k+1)^4}{\lambda^{(k+4)/2}}$, so we deduce from \eqref{e:c}-\eqref{e:c'''} that $c^{(4)}(\lambda) = O(\lambda^{-2})$. Since $w=(q+1)c+\alpha s$, we have showed that
\[
w^{(4)}(\lambda) = O(\lambda^{-2})\,.
\]

This proves that $\frac{\phi_n^{(4)}(\theta)}{\phi_n''(0)}$ stays uniformly bounded, so $F'' = O(n^{-2})$ as desired. 

Recalling \eqref{e:q11'2}, we have shown that $V_{0,\frac{\pi}{2}}(Q_{1,1}) = \int_0^{\pi/2}|Q_{1,1}'(\theta)|\,\dd\theta \le \frac{\pi}{2}\|Q_{1,1}'\|_{\infty}$ is summable. This finally completes the proof of \eqref{e:mainclaimdiag}, i.e. the sums of error terms provided by Corollary~\ref{cor:phase} are indeed $O(t^{-1})$.

\subsection{Handling one edge}\label{sec:1ed}
As discussed in \cite[Appendix]{ISW}, for $x,y$ in the same edge, it follows from \eqref{e:diffsame} that
\[
\Im G^\lambda_{\mathbf{T}_q}(x,y)= \Psi_1(\lambda)\Psi_2(\lambda)\,,
\]
where $\Psi_1(\lambda) = \Im G_{\mathbf{T}_q}^\lambda(o,o)=(-1)^{n+1}s(\lambda)\Psi(w(\lambda))$ on $I_n$ as before and
\begin{multline}\label{e:psi2}
\Psi_2(\lambda)\\
=\frac{S_{\lambda}(L-x)S_{\lambda}(L-y)+S_{\lambda}(x)S_{\lambda}(y)+\frac{w(\lambda)}{q+1}[S_{\lambda}(L-x)S_\lambda(y)+S_{\lambda}(x)S_{\lambda}(L-y)]}{s(\lambda)^2}.
\end{multline}
In view of \eqref{e:contker}, noting that $\Psi_1(\pm 2\sqrt{q})=0$, we have
\begin{multline}\label{e:expan2}
\ee^{\ii tH}\mathbf{1}_{ac}(H)(x,y) = \frac{-1}{\ii \pi t}\int_{\sigma_{ac}(H)}\ee^{\ii t\lambda}(\Psi_1\Psi_2)'(\lambda)\,\dd\lambda  \\
=\frac{-\ii}{\pi t}\sum_{n=1}^{\infty}(-1)^n\int_{I_n}\ee^{\ii t \lambda}\Big[\frac{\left\{s(\lambda)\Psi_2(\lambda)\right\}'\Psi(w(\lambda))}{w'(\lambda)} + s(\lambda)\Psi_2(\lambda)\Psi'(w(\lambda))\Big]w'(\lambda)\,\dd\lambda\,.
\end{multline}

We now repeat the arguments of \S~\ref{sec:diag}~:

\medskip

\textbf{Step 1~:} The analog of \eqref{e:someeq} is still $O(t^{-1})$. To prove this, we need to check that $\frac{\dd}{\dd\theta}\frac{\sqrt{q}(q+1)\sin\theta[s'(\phi_n(\theta))\Psi_2(\phi_n(\theta)+s(\phi_n(\theta))\Psi_2'(\phi_n(\theta))]}{(q+1)^2-4q\cos^2\theta}=O(n^{-2})$. In turn, it suffices to prove that $s'(\phi_n(\theta))\Psi_2(\phi_n(\theta))$, $s(\Phi_n(\theta))\Psi_2'(\phi_n(\theta))$ and their $\theta$-derivatives are $O(n^{-2})$. For this, we note that for $j\le 2$,
\begin{equation}\label{e:psi2boun}
\Psi_2^{(j)}(\lambda) = O(\lambda^{-j/2})\,.
\end{equation}
In fact, expansion \eqref{e:sersin} is valid more generally for $S_{\lambda}(x)$, replacing $L$ by $x$. Using \eqref{e:sdeca} and its analog, we see that each term defining $\Psi_2(\lambda)$ is $\frac{O(\lambda^{-1})}{\frac{\sin^2\sqrt{\lambda}L}{\lambda}+O(\lambda^{-3/2})}$, and we showed after \eqref{e:w'} that $\sin\sqrt{\lambda}L$ is uniformly bounded from below over $\sigma_{ac}(H)$. Recalling that $|w(\lambda)|\le 2\sqrt{q}$ on $\sigma_{ac}(H)$, we conclude that $\Psi_2(\lambda)=O(1)$.

\medskip

Next, $\Psi_2'(\lambda)$ consists of terms of the form $\frac{S_{\lambda}'(x)S_{\lambda}(y)}{s^2(\lambda)}$, $\frac{S_{\lambda}(x)S_{\lambda}(y)s'(\lambda)}{s^3(\lambda)}$ and $\frac{w'(\lambda)S_{\lambda}(x)S_{\lambda}(y)}{s^2(\lambda)}$. Each of these terms is $O(\lambda^{-1/2})$ in view of \eqref{e:w'} and the analog of \eqref{e:sdeca} with $L$ replaced by $x$. This proves \eqref{e:psi2boun} for $j=1$. Similarly, using \eqref{e:w''} and the analog of \eqref{e:s''} for any $x$, we see that \eqref{e:psi2boun} holds for $j=2$.

\medskip

We may now conclude the main claim of this step~: we get that $s'(\phi_n(\theta))\Psi_2(\phi_n(\theta))$ and $s(\phi_n(\theta))\Psi_2'(\phi_n(\theta))$ are both $O(n^{-2})$ (recall that $\phi_n(\theta)\asymp n^2$). The derivatives have the form $s''(\phi_n(\theta))\Psi_2(\phi_n(\theta))\phi_n'(\theta)$, $s'(\phi_n(\theta))\Psi_2'(\theta)\phi_n'(\theta)$ and $s(\phi_n(\theta))\Psi_2''(\phi_n(\theta))\phi_n'(\theta)$. Each term is $O(n^{-2})$ as required.

\medskip

\textbf{Step 2~:} We now turn to the main contribution in \eqref{e:expan2}, $\int_{I_n}s(\lambda)\Psi_2(\lambda)\Psi'(w(\lambda))w'(\lambda)$. Clearly in the main term of \eqref{e:phasresult}, one just replaces $s(a_n)$ by $s(a_n)\Psi_2(a_n)$ and $s(b_n)$ by $s(b_n)\Psi_2(b_n)$. It remains to control the error term. Following the arguments, we arrive at \eqref{e:q11'2}, where now $F(\theta) = w'(\phi_n(\theta))s(\phi_n(\theta))\Psi_2(\phi_n(\theta))g(\theta)\sqrt{1+\frac{\phi_n'''(\theta_0)}{3\phi_n''(0)}\theta}$ and $S(\theta)=w'(\phi_n(\theta))s(\phi_n(\theta))\Psi_2(\phi_n(\theta))g(\theta)(\sin\theta-\theta)$. The first two terms in \eqref{e:q11'2} are controlled by showing that $F'$, $F''$ are $O(n^{-2})$. This is indeed the case using \eqref{e:psi2boun} (and the bounds provided for \eqref{e:F'}-\eqref{e:F''}). The last two terms become $\frac{(w''sg\Psi_2+w's'g\Psi_2+w's\Psi_2'g)\phi_n'(\theta)+w's\Psi_2g'}{2\sqrt{q}}\cdot\frac{\sin\theta-\theta}{\theta\sin\theta} - \frac{w'sg\Psi_2}{2\sqrt{q}}\cdot\frac{\sin^2\theta-\theta^2\cos\theta}{\theta^2\sin^2\theta}$. This is again $O(n^{-2})$ in view of \eqref{e:psi2boun} and our bounds on $w^{(k)}$ and $s^{(k)}$.

\medskip

\textbf{Conclusion~:} We conclude that if $x,y\in\mathbf{T}_q$ belong to the same edge $e$, then
\begin{multline*}
\ee^{\ii tH}\mathbf{1}_{ac}(H)(x,y)= \frac{\ii g(0)}{\sqrt{\pi \sqrt{q}}\, t^{3/2}}\sum_{n\ge 1}(-1)^n\bigg[ \frac{\ee^{\frac{\ii \pi}{4}}\ee^{\ii a_n t}|w'(a_n)|^{1/2}s(a_n)\Psi_2(a_n) }{2}\\
- \frac{\ee^{\frac{-\ii \pi}{4}}\ee^{\ii b_n t}|w'(b_n)|^{1/2}s(b_n)\Psi_2(b_n)}{2}\bigg]+O(t^{-2})
\end{multline*}
independently of the edge $e$, where $\Psi_2(\lambda)=\Psi_2(\lambda,x,y)$ is given by \eqref{e:psi2}. The error $O(t^{-2})$ is independent of $x,y\in e$, since the bound \eqref{e:psi2boun} is uniform over $x,y$ belonging to the compact interval $[0,L]$. Note that $\Psi_2(\lambda,0,0)=1$, in which case the above expression coincides with that of \S~\ref{sec:diag}, as it should.

\subsection{Handling distinct edges}
The last case is when $x\in e_1$ and $y\in e_2$ for distinct edges. Then there is a unique path $(v_0,\dots,v_k)$ from $e_1$ to $e_2$, so that $e_1=(v_0,v_1)$ and $e_2=(v_{k-1},v_k)$. It is shown in \cite[eq. (5.4),(5.6)]{ISW} that in this case
\[
\Im G_{\mathbf{T}_q}^\lambda(x,y) = \Psi_1(\lambda)\Psi_3(\lambda)\,,
\]
where $\Psi_1(\lambda)=\Im G_{\mathbf{T}_q}^\lambda(o,o) = (-1)^{n+1}s(\lambda)\Psi(w(\lambda))$ as before and $\Psi_3(\lambda)=\Psi_3(\lambda,x,y)$ is given by
\begin{multline}\label{e:psi3}
\Psi_3(\lambda)=\frac{S_{\lambda}(L-x)S_{\lambda}(y)}{s^2(\lambda)}\Phi_k(w(\lambda)) \\
+ \frac{S_{\lambda}(L-x)S_{\lambda}(L-y)+S_{\lambda}(x)S_{\lambda}(y)}{s^2(\lambda)}\Phi_{k-1}(w(\lambda)) + \frac{S_{\lambda}(x)S_{\lambda}(L-y)}{s^2(\lambda)}\Phi_{k-2}(w(\lambda))\,.
\end{multline}

As before, we get
\begin{multline}\label{e:expan3}
\ee^{\ii tH}\mathbf{1}_{ac}(H)(x,y)  \\
=\frac{-\ii}{\pi t}\sum_{n=1}^{\infty}(-1)^n\int_{I_n}\ee^{\ii t \lambda}\Big[\frac{\left\{s(\lambda)\Psi_3(\lambda) \right\}'\Psi(w(\lambda))}{w'(\lambda)} + s(\lambda)\Psi_3(\lambda)\Psi'(w(\lambda))\Big]w'(\lambda)\,\dd\lambda
\end{multline}
and repeat the arguments of \S~\ref{sec:1ed}. In the present case, we will be concerned not only with the $\lambda$-decay (i.e. summability in $n$), but also about having errors uniform in $k$ (uniformity in $x,y\in [0,L]$ is clear from the definition of $\Psi_3$).

\medskip

\textbf{Step 1~:} We study $s'(\phi_n(\theta))\Psi_3(\phi_n(\theta))$, $s(\phi_n(\theta))\Psi_3'(\phi_n(\theta))$ and their $\theta$-derivatives. We already know that the coefficients of $\Phi_m(w(\lambda))$ in \eqref{e:psi3}, consisting of quotients of $S_{\lambda}$ functions, have their $j$-th $\lambda$-derivatives behaving like $O(\lambda^{-j/2})$ for $0\le j\le 2$, see \eqref{e:psi2boun}. At the energy $\lambda=\phi_n(\theta)$, they are thus $O(n^{-j})$. It remains to analyze $\frac{\dd^j}{\dd\lambda^j}\Phi_m(w(\lambda)) = \Phi_m^{(j)}(w(\lambda))w^{(j)}(\lambda)$. More precisely, since we study $\Psi_3(\phi_n(\theta))$, we would like to control $\Phi_m^{(j)}(w(\phi_n(\theta)))w^{(j)}(\phi_n(\theta)) = \Phi_m^{(j)}(2\sqrt{q}\cos\theta)w^{(j)}(\phi_n(\theta))$. But by \eqref{e:sfer},
\[
\Phi_m(2\sqrt{q}\cos\theta)=q^{-m/2}\bigg(\frac{2}{q+1} \cos (m\theta) + \frac{q-1}{q+1}\frac{\sin(m+1)\theta}{\sin \theta}\bigg)\,,
\]
so we clearly have $|\Phi_m(2\sqrt{q}\cos\theta)|\le \frac{q^{-m/2}}{q+1}(2+(q-1)(m+1)) = O(1)$, with $O(1)$ independent of $m$. Similarly, we showed after \eqref{e:cheby2} that $|\Phi_m'(s)|\le \frac{m^2q^{-m/2}}{2\sqrt{q}(q+1)}(2+(q-1)(m+1)) = O(1)$ for any $s\in [-2\sqrt{q},2\sqrt{q}]$, where again we have $O(1)$ independent of $m$ in view of the fast decay in $m$. Finally $\Phi_m''(\lambda) = \frac{q^{-m/2}}{4q(q+1)}\Big(2P_m''\Big(\frac{\lambda}{2\sqrt{q}}\Big)+(q-1)Q_m''\Big(\frac{\lambda}{2\sqrt{q}}\Big)\Big)$. Now $P_m'' = mQ_{m-1}'$, so $|P_m''|\le m^2(m-1)^2$ using \eqref{e:cheby2}. Using \eqref{e:cheby2} again, we also have $|Q_m''|\le (m+1)m^2(m-1)^2$. Thus, $|\Phi_m''(s)|\le \frac{m^4q^{-m/2}}{4q(q+1)}(2+(q-1)(m+1))$, which is still $O(1)$ uniformly in $m$. Recalling our bounds over $w^{(j)}$, it follows that for $j\le 2$,
\begin{equation}\label{e:Fider}
\Phi_m^{(j)}(2\sqrt{q}\cos\theta)w^{(j)}(\phi_n(\theta)) = O(n^{-j})\,,
\end{equation}
uniformly in $m$. Using \eqref{e:psi2boun}, we deduce that
\begin{equation}\label{e:psi3boun}
\Psi_3^{(j)}(\phi_n(\theta)) = O(n^{-j})
\end{equation}
uniformly in $k=d(e_1,e_2)$, for $j\le 2$.

It follows from \eqref{e:sdeca} and \eqref{e:psi3boun} that $s'(\phi_n(\theta))\Psi_3(\phi_n(\theta))$ and $s(\phi_n(\theta))\Psi_3'(\phi_n(\theta))$ are $O(n^{-2})$. Also, $\frac{\dd}{\dd\theta}s'(\phi_n(\theta))\Psi_3(\phi_n(\theta)) = [s''(\phi_n(\theta))\Psi_3(\phi_n(\theta))+s'(\phi_n(\theta)\Psi_3'(\phi_n(\theta))]\phi_n'(\theta)$ is $O(n^{-2})$ if we recall \eqref{e:s''} and the bound $\phi_n'(\theta)=O(n)$. Finally $\frac{\dd}{\dd\theta}s(\phi_n(\theta))\Psi_3'(\phi_n(\theta)) = [s'(\phi_n(\theta))\Psi_3'(\phi_n(\theta))+s(\phi_n(\theta))\Psi_3''(\phi_n(\theta))]\phi_n'(\theta)=O(n^{-2})$ by the same identities. We conclude that the analog of \eqref{e:someeq} is still $O(t^{-1})$, uniformly in $k=d(e_1,e_2)$.

\medskip

\textbf{Step 2~:} We next study $\int_{I_n}\ee^{\ii t\lambda}s(\lambda)\Psi_3(\lambda)\Psi'(w(\lambda))w'(\lambda)\,\dd\lambda$, to which we apply the stationary phase method. Again in the main term in \eqref{e:phasresult}, we just replace $s(a_n)$ and $s(b_n)$ by $s(a_n)\Psi_3(a_n)$ and $s(b_n)\Psi_3(b_n)$, respectively. The error terms are exactly as in the previous subsection, with $\Psi_2$ replaced by $\Psi_3$. Using \eqref{e:psi3boun} instead of \eqref{e:psi2boun}, the same argument shows that these errors are $O(n^{-2})$. We conclude that
\begin{multline*}
\ee^{\ii tH}\mathbf{1}_{ac}(H)(x,y) = \frac{\ii g(0)}{\sqrt{\pi \sqrt{q}}\, t^{3/2}}\sum_{n\ge 1} (-1)^n\bigg[\frac{\ee^{\frac{\ii \pi}{4}}\ee^{\ii a_n t}|w'(a_n)|^{1/2}s(a_n)\Psi_3(a_n)}{2}\\
 - \frac{\ee^{\frac{-\ii \pi}{4}}\ee^{\ii b_n t}|w'(b_n)|^{1/2}s(b_n)\Psi_3(b_n)}{2}\bigg]+O(t^{-2})
\end{multline*}
with $O(t^{-2})$ uniform in $k=d(e_1,e_2)$. This completes the proof of the main result on the kernel (recall that $(-1)^{n+1}s(\lambda)=|s(\lambda)|$ on $I_n$ as mentioned after \eqref{e:psi1}).

\medskip

Finally, the series above is absolutely convergent. In fact, using \eqref{e:sdeca}, \eqref{e:w'} and \eqref{e:psi2boun} or \eqref{e:psi3boun} (with $j=0$) according to the positions of $x,y$, we see the general term is $O(n^{-3/2})$ in absolute value, independently of $t$ or $k$. Consequently, as $t\To\infty$, we have $|\ee^{\ii tH}\mathbf{1}_{ac}(H)(x,y)|\le \frac{C_q}{t^{3/2}}$ uniformly in $x,y\in\mathbf{T}_q$. The dispersive estimate follows (see Remark~\ref{rem:kernel}).

\appendix

\section{Stationary phase result}\label{app:stat}

In this appendix we give an explicit stationary phase estimate. To put things into perspective, we need a stronger statement than the Van der Corput Lemma \cite[Corollary, p.334]{Stein} in the sense that we want an asymptotic $\sim$ for the principal term, but we accept a weaker statement than the full asymptotic given in \cite[Proposition 3, p.334]{Stein} as we only care about the principal term. Our point is to make the remainder explicit in the phase function and observable, with as few derivatives as possible. This is important for our applications for quantum graphs, where we need to apply this for a series of integrals so we have to ensure the series of errors converge. For this we shall use the explicit version of \cite{Olv}.

\begin{thm}
Let $p\in C^1[a,b]$, $q \in C[a,b]$ and suppose $p$ and $q$ admit a Taylor expansion at $x=a$. Assume $x_0=a$ is the only critical point of $p$ in $[a,b]$, so $p'(a)=0$ and $p'(x)\neq 0$ for $x\in (a,b]$. Assume moreover $p''(a)\neq 0$ and let $\eps=\sgn p''(a)$. Then
\[
\int_a^b\ee^{\ii t p(x)}q(x)\,\dd x = \ee^{\ii t p(a)}\ee^{\eps \pi\ii/4}\sqrt{\frac{\pi}{2\eps p''(a)t}}q(a) + \delta_{0,1}(t)-\veps_{0,1}(t)\,,
\]
where
\[
\delta_{0,1}(t) = \int_a^b\ee^{\ii tp(x)}\bigg(q(x)-\frac{\eps q(a)p'(x)}{\sqrt{2\eps p''(a)}\sqrt{\eps(p(x)-p(a))}}\bigg)\,\dd x\,,
\]
\[
\veps_{0,1}(t)=\ee^{\ii tp(a)}\ee^{\eps \pi\ii/4}\Gamma\Big(\frac{1}{2},\ii t\eps(p(a)-\ii p(b))\Big)\frac{q(a)}{\sqrt{2\eps p''(a) t}}
\]
and $\Gamma(\alpha,z) = \int_z^\infty \ee^{-t}t^{\alpha-1}\,\dd t$.
\end{thm}
The statement can be greatly generalized : one can replace $p''(a)\neq 0$ by $p^{(k)}(a)\neq 0$, where $p^{(j)}(a)=0$ for all $j<k$. The function $q(x)$ can also have an algebraic singularity $(x-a)^{-\rho}$, $0\le \rho <1$. Finally higher order precision is also available with explicit terms.
\begin{proof}
First assume $p'(x)>0$ for all $x\in (a,b]$. We apply \cite[Theorem 1]{Olv} with $\lambda=1$, $\mu=2$, $m=0$, $n=1$. Since $p,q$ admit Taylor expansions at $x=a$, we have $p(x)=p(a)+\sum_{s=0}^\infty p_s(x-a)^{s+2}$ and $q(x)=\sum_{s=0}^\infty q_s(x-a)^s$ with $p_0 = \frac{p''(a)}{2}$, $q_0=q(a)$. Let $a_0 = \frac{q_0}{2\sqrt{p_0}}$. Noting that $n>m\mu-\lambda$, we take $\nu=1$ and get
\[
\int_a^b\ee^{\ii t p(x)}q(x)\,\dd x  = \ee^{\ii t p(a)}\ee^{\pi\ii/4}\Gamma\Big(\frac{1}{2}\Big)\frac{a_0}{t^{1/2}}+\delta_{0,1}(t)-\veps_{0,1}(t),
\]
with $\delta_{0,1}(t)= \int_a^b \ee^{\ii t p(x)}Q_{0,1}'(x)\,\dd x$ and $Q_{0,1}(x) = \int_0^xq(y)\,\dd y - \frac{\Gamma(\frac{1}{2})}{\Gamma(\frac{3}{2})}a_0\sqrt{p(x)-p(a)}$, and $\veps_{0,1}(t)=\ee^{\ii tp(a)}\ee^{\pi\ii/4}\Gamma(\frac{1}{2},\ii tp(a)-\ii tp(b))\frac{q_0/2\sqrt{p_0}}{\sqrt{t}}$. The statement follows when $\sgn p''(a)>0$. Indeed, expanding $p'(x)= p''(a)(x-a)(1+O(x-a))$, we see that that $p'(x)>0$ iff $\sgn p''(a)>0$.

\medskip

Now assume $p'(x)<0$ for all $x\in (a,b]$, implying $\sgn p''(a)<0$. As remarked in \cite{Olv}, the theorem remains true by essentially replacing $\ii$ by $-\ii$ through most of the proof. More precisely, here $p(x)$ is decreasing so with the notations of \cite{Olv}, one considers the change of variables $v=p(a)-p(x)$ instead. Then $\int_a^b \ee^{\ii tp(x)}q(x)\,\dd x = \ee^{\ii t p(a)}\int_0^{p(a)-p(b)}\ee^{-\ii t v}f(v)\,\dd v$, with $f(v)=\frac{q(x)}{-p'(x)}$. This shows all $p_s$ get replaced by $-p_s$. Moreover, for (5.4) to hold, we should take $P_j(x)=\left\{\frac{-1}{p'(x)}\frac{\dd}{\dd x}\right\}^j\frac{q(x)}{-p'(x)}$. We take $\beta=p(a)-p(b)$. Lemmas 1,2,3 in \cite[Section 4]{Olv} continue to hold verbatim if we replace $\ii$ by $-\ii$ on both hypothesis and conclusion (e.g. now $\ds \lim_{\eta\downarrow 0}\int_\beta^\infty \ee^{-\eta v}\ee^{-\ii tv}v^{\alpha-1}\,\dd v = \frac{\ee^{-\alpha\pi\ii/2}}{x^\alpha}\Gamma(\alpha,\ii t\beta)$). Returning to $\int_0^\beta\ee^{-\ii tv}f(v)\,\dd v$, we see that $\ii$ should be replaced by $-\ii$ everywhere in (5.8)--(5.11). The same replacement holds for (5.12) and (5.13), except for the terms containing $\veps,\delta$, i.e. we have $\ee^{-\ii tp(a)}\veps_{m,n}(t)$ and $\ee^{-\ii tp(a)}\{\delta_{m,n}(t)-\veps_{m,n}(t)\}$, respectively. (5.14) becomes $\ee^{\ii tp(a)}(\frac{-\ii}{t})^m\int_0^\beta \ee^{-\ii tv}\phi_n^{(m)}(v)\,\dd v$. With our choice of $P_j$, $\phi_n^{(m)}$ has the required form, completing the proof.
\end{proof}

The following corollary is the main tool we use instead of the Van der Corput Lemma, cf. \cite[Corollary, p.334]{Stein}, to obtain sharp estimates.

\begin{cor}\label{cor:phase}
Under the assumptions of the previous theorem, define
\begin{equation}\label{e:q11}
Q_{1,1}(x) = \frac{q(x)}{\eps p'(x)} - \frac{q(a)}{\sqrt{2\eps p''(a)}\sqrt{\eps (p(x)-p(a))}}\,.
\end{equation}
Then
\begin{multline*}
\bigg|\int_a^b\ee^{\ii tp(x)}q(x)\,\dd x - \ee^{\ii t p(a)}\ee^{\eps \pi\ii/4}\sqrt{\frac{\pi}{2|p''(a)|t}}q(a)\bigg| \\ \le \frac{1}{t}\bigg(| Q_{1,1}(a)| + |Q_{1,1}(b)| + V_{a,b}(Q_{1,1}) + \frac{2|q(a)|}{\sqrt{2|p''(a)|}\sqrt{|p(b)-p(a)|}}\bigg)\,,
\end{multline*}
where $V_{a,b}(Q_{1,1}) = \int_a^b|Q_{1,1}'(y)|\,\dd y$ is the total variation of $Q_{1,1}$ over $[a,b]$.
\end{cor}
\begin{proof}
Apply \cite[eq. (6.3), (6.7)]{Olv} to the previous theorem.
\end{proof}

\begin{exa}
As is well-known, for any $\alpha\in \R$, the Fresnel integral $\int_0^\infty \ee^{\ii t\alpha x^2}\,\dd x = \frac{\ee^{\eps\pi\ii/4}}{2}\sqrt{\frac{\pi}{|\alpha|\,t}}$, where $\eps = \sgn \alpha$. The previous result tells us that if we cutoff at $A>0$, then
\[
\bigg|\int_0^A\ee^{\ii t\alpha x^2}-\frac{\ee^{\eps\pi\ii/4}}{2}\sqrt{\frac{\pi}{|\alpha|\,t}}\bigg|\le \frac{1}{A|\alpha|\,t}\,.
\]
Indeed, here $Q_{1,1}\equiv 0$.
\end{exa}

In general, we should compute the limit $Q_{1,1}(a)$ carefully. Say $\eps=1$. Then $Q_{1,1}(x) = \frac{q(x)\sqrt{2p''(a)}\sqrt{p(x)-p(a)}-q(a)p'(x)}{p'(x)\sqrt{2p''(a)}\sqrt{p(x)-p(a)}}$. Expanding $p(x)=p(a)+\frac{p''(a)}{2}(x-a)^2+\frac{p'''(a_x)}{6}(x-a)^3$, also $p'(x)=p''(a)(x-a)+\frac{p'''(\tilde{a}_x)}{2}(x-a)^2$ and $q(x) = q(a)+q'(\hat{a}_x)(x-a)$, for some $a_x,\tilde{a}_x,\hat{a}_x\in (a,x)$, the numerator becomes
\begin{multline*}
q(x)\bigg[p''(a)(x-a)\sqrt{1+\frac{p'''(a_x)}{3p''(a)}(x-a)}\bigg]-q(a)p'(x)\\
= \left[q(a)+q'(\hat{a}_x)(x-a)\right]\cdot p''(a)(x-a)\cdot \bigg[1+\frac{p'''(a_x)}{6p''(a)}(x-a)+O(x-a)^2\bigg]\\
-q(a)\Big[p''(a)(x-a)+\frac{p'''(\tilde{a}_x)}{2}(x-a)^2\Big]\\
= q(a)\frac{p'''(a_x)}{6}(x-a)^2+O(x-a)^3+q'(\hat{a}_x)p''(a)(x-a)^2-\frac{q(a)p'''(\tilde{a}_x)}{2}(x-a)^2
\end{multline*}
while the denominator is $[p''(a)(x-a)+O(x-a)^2][p''(a)(x-a)\sqrt{1+O(x-a)}]$. Thus,
\begin{equation}\label{e:q11a}
Q_{1,1}(a) = \frac{q(a)\frac{p'''(a)}{6}+q'(a)p''(a)-\frac{q(a)p'''(a)}{2}}{p''(a)^2} = \frac{q'(a)}{p''(a)}-\frac{q(a)p'''(a)}{3p''(a)^2}\,.
\end{equation}
The same calculation shows that in general $Q_{1,1}(a) = \eps\big(\frac{q'(a)}{p''(a)}-\frac{q(a)p'''(a)}{3p''(a)^2}\big)$.

We thus have in all cases
\begin{multline}\label{e:corprec}
\bigg|\int_a^b\ee^{\ii tp(x)}q(x)\,\dd x - \ee^{\ii t p(a)}\ee^{\eps \pi\ii/4}\sqrt{\frac{\pi}{2|p''(a)|t}}q(a)\bigg| \\ \le \frac{1}{t}\bigg(\Big|\frac{q'(a)}{p''(a)}-\frac{q(a)p'''(a)}{3p''(a)^2}\Big| + \Big|\frac{q(b)}{p'(b)}\Big| + V_{a,b}(Q_{1,1}) + \frac{3|q(a)|}{\sqrt{2|p''(a)|}\sqrt{|p(b)-p(a)|}}\bigg)\,.
\end{multline}

If the only critical point is at $x=b$ instead, then via the change of variables $y=-x$, $\tilde{p}(y)=p(-y)$ and $\tilde{q}(y)=q(-y)$, we see that
\begin{multline}\label{e:end}
\bigg|\int_a^b\ee^{\ii tp(x)}q(x)\,\dd x - \ee^{\ii t p(b)}\ee^{\eps \pi\ii/4}\sqrt{\frac{\pi}{2|p''(b)|t}}q(b)\bigg| \\ \le \frac{1}{t}\bigg(\Big|\frac{q'(b)}{p''(b)}-\frac{q(b)p'''(b)}{3p''(b)^2}\Big| + \Big|\frac{q(a)}{p'(a)}\Big| + V_{a,b}(\widetilde{Q}_{1,1}) + \frac{3|q(b)|}{\sqrt{2|p''(b)|}\sqrt{|p(b)-p(a)|}}\bigg)\,,
\end{multline}
where $\eps = \sgn p''(b)$ and $\widetilde{Q}_{1,1}(x) = \frac{q(x)}{\eps p'(x)} - \frac{q(b)}{\sqrt{2\eps p''(b)}\sqrt{\eps (p(x)-p(b))}}$.

\begin{rem}
We conclude this appendix by comparing our statement (which is just a streamlined account of \cite{Olv}) with some classical references.
\begin{enumerate}[(1)]
\item The well-known Van der Corput lemma \cite[Corollary p.334]{Stein} gives some explicit bound over $\big|\int_a^b\ee^{\ii t p(x)}q(x)\,\dd x \big|$. However, this only yields an upper bound, not an asymptotically equivalent term. Moreover, it requires the additional condition $|p''(x)|\ge 1$ over $[a,b]$.
\item The full asymptotics given in \cite[p.334]{Stein} or \cite[p.41]{Zwor} do not have explicit bounds on the remainder. Inspecting the proof of \cite{Stein}, one first restricts the integral to a neighborhood $N_{\eps}(a)$ of the critical point $a$ such that $\left|\frac{p'''(x)}{3p''(a)}(x-a)\right|<1$. The remainder integral is estimated using integration by parts. This means one needs to control the size of $N_\eps(a)$ and have a lower bound over $p'(x)$ outside $N_\eps(a)$. A similar requirement appears in the proofs of \cite[p.41, p.45]{Zwor}. When such information is available one can expect to control $V_{a,b}(Q_{1,1})$ in \eqref{e:corprec} by $(b-a)\|Q_{1,1}'\|_{\infty}$ efficiently; this is in fact what we did in the discussion following \eqref{e:reducedF}.
\item The methods of \cite{Stein,Zwor} seem more costly in terms of derivatives. After a change of variables $y=T(x)$, where $T$ depends on the phase function $p$, the reduced integral (within the neighborhood) becomes $\int_{N_{\eps}(a)}\ee^{\ii tp(x)} q(x)\,\dd x = \int_{T(N_\eps(a))} \ee^{\ii t \eps y^2} u(y)\,\dd y$, where $u(y)=\frac{q(T^{-1}y)}{|T'(T^{-1}y)|}$. It is now necessary to control the derivatives of $u$. In fact a bound on the error we could extract from \cite[Step 2, p.335]{Stein} with $x\eta(x):=u(x)-u(a)$ is $Ct^{-1}|u|_2$, where $|u|_2 = \max(\|u'\|_\infty,\|u''\|_{\infty})$. The first method of \cite[p.43]{Zwor} is more costly, requiring bounds over $\|u^{(k)}\|_{\infty}$ for $k\le 4$. However, after involved Taylor-Lagrange expansions, one sees that $\|u^{(k)}\|_{\infty}$ is controlled by 
$\ds \max_{\substack{j\le k\\ \ell\le k+3}} \frac{\|q^{(j)}p^{(\ell)}\|_{\infty}}{|p''(a)|^2}$. This means that we need to control at least $5$ derivatives of $p$, and $2$ derivatives of $q$.

The second method in \cite[p.45]{Zwor} seems more costly for the observable. Taking $m=1$ and controlling the error $I_h^{(1)}(0)$ by taking $N=2$ in \cite[p.43]{Zwor}, one finds it necessary to control all derivatives $\frac{\|(g_pq)^{(k)}\|_{\infty}}{|p''(a)|^2}$ for $k\le 6$, where $g_p(x) = p(x)-p''(0)x^2/2$, here $a=0$ and $p(0)=0$.
\end{enumerate}
\end{rem}

\section{Brief comparison with 1d periodic potentials}\label{app:per}

We continue here the discussion started in \S~\ref{sec:q=1}.

For transparency, consider first $H=-\Delta$ on $\R$. The Green's function $G^z_{\R}(x,y)$ for $z\in \C^+$ can be constructed as usual using two semi-$L^2$ ODE solutions. For example take $V_z(x)=\ee^{\ii\sqrt{z}x}\in L^2[0,\infty)$ and $U_z(x)=\ee^{-\ii\sqrt{z}x}\in L^2(-\infty,0]$. Their Wronskian $V_zU_z'-V_z'U_z=-2\ii\sqrt{z}$, so $G^z(x,y) = \begin{cases} \frac{\ee^{\ii\sqrt{z}x}\ee^{-\ii\sqrt{z}y}}{-2\ii\sqrt{z}} & \text{if } y\le x,\\ \frac{\ee^{\ii\sqrt{z}y}\ee^{-\ii\sqrt{z}x}}{-2\ii\sqrt{z}}& \text{if } y\ge x. \end{cases}$ Hence, $G^z(x,y) = \frac{\ee^{\ii\sqrt{z}|x-y|}}{-2\ii\sqrt{z}}$ and $\Im G^{\lambda+\ii 0}(x,y)= \frac{\cos\sqrt{\lambda}|x-y|}{2\sqrt{\lambda}}$ for $\lambda\in [0,\infty)=\sigma(-\Delta)$. By the spectral theorem, $\ee^{\ii tH}(x,y) = \frac{1}{\pi}\int_{\sigma(H)}\ee^{\ii t\lambda}\Im G^{\lambda+\ii 0}(x,y)\,\dd \lambda=\int_0^\infty \frac{\ee^{\ii t\lambda}\cos\sqrt{\lambda}|x-y|}{2\pi \sqrt{\lambda}}\,\dd\lambda = \int_0^\infty \frac{\ee^{\ii t(\lambda+\frac{\sqrt{\lambda}|x-y|}{t})}+\ee^{\ii t (\lambda-\frac{\sqrt{\lambda}|x-y|}{t})}}{4\pi \sqrt{\lambda}}\,\dd\lambda$. Denote the velocity $v=\frac{|x-y|}{t}$ and consider the changes of variables $\sqrt{\lambda}=k$ and $-\sqrt{\lambda}=k$, respectively. Then $\ee^{\ii t H}(x,y) = \int_0^\infty \frac{\ee^{\ii t(k^2+kv)}}{2\pi}\,\dd k + \int_{-\infty}^0\frac{\ee^{\ii t(k^2+kv)}}{2\pi}\,\dd k = \frac{1}{2\pi}\int_{-\infty}^\infty \ee^{\ii t(k^2+kv)}\,\dd k$. This is a Fresnel-type integral, it reduces to $\ee^{\ii tH}(x,y)=\frac{1}{2}\sqrt{\frac{\ii}{\pi t}}\ee^{\frac{-\ii tv^2}{4}}$.

For general periodic Schr\"odinger operators $H$ on $\R$, one simply replaces $\ee^{\pm \ii \sqrt{z}}$ by Floquet solutions. The spectrum generally consists of a number of bands $(I_n)_{n\ge 1}$ of purely absolutely continuous spectrum which may be finite or infinite. A corresponding variable $k$ is defined mapping $I_n$ to bands $\Sigma_n\cup (-\Sigma_{n})=:\Sigma(n)$, and one finds that $\ee^{\ii tH}(x,y) = \ds \sum_n\int_{\Sigma(n)} \ee^{\ii t(E(k)-kv)} X^+(x,k)X^-(x,k)\,\dd k$, here $E(k)$ behaves like $k^2$ away from the band edges and $X^{\pm}$ come from the Floquet solution. See \cite{Fir,Cucca} for more details. This integral is now analyzed using the stationary phase method. It was shown in \cite[Corollary 4.4]{Kor} that for finite bands $\Sigma_n$, $E''(k)$ has a unique zero $k_n\in \Sigma_n$, moreover $E'(k)$ is monotone increasing up to $k_n$, then monotone decreasing. Now consider the phase function $\phi(k) = E(k)-kv$. We have $\phi'(k)=E'(k)-v$ and $\phi''(k)=E''(k)$. The only possibility that $\phi'(k)=\phi''(k)=0$, $k\in\Sigma_n$, is if $k=k_n$ and $v=E'(k_n)$, i.e. for a very specific choice of $v$, hence $x,y,t$. In this situation the stationary phase method allows to conclude that the speed of dispersion slows down to $t^{-1/3}$ (or slower in principle). This problem does not arise on the infinite band. In all other cases the decay will be $t^{-1/2}$, or even faster when $v$ is very large (in that case $\phi'(k)$ does not vanish). See \cite{Fir} for details when the number of bands is finite. The paper \cite{Cucca} considers the case of infinite number of bands, but only provides upper bounds, as it relies on the Van der Corput lemma; it can be an interesting question to test for sharpness by providing asymptotic equivalents as in \cite{Fir}.

Back to our case of quantum trees, the idea of constructing the resolvent kernel from two semi-$L^2$ functions works again, see \cite{Carl,ISW}. The Floquet functions $\ee^{\pm \ii k x}X^{\pm}(x,k)$ are replaced by $(\mu^-(\lambda)^m)V^+_\lambda(x)$ and $(\mu^-(\lambda))^mU^-_{\lambda}(x)$, where $m$ is the distance of $x$ from a fixed edge $b_0$ (think of $b_0=[0,1]$ in $\R$) and $V^+_\lambda,U_\lambda^-$ are fixed functions repeated on all edges (i.e. can be regarded as periodic). The main difference is that the multiplicative factor $\mu^-(\lambda)^m$ decays exponentially in $m$, in fact $|\mu^-(\lambda)^m|=q^{-m/2}$, in contrast to $|\ee^{\ii k m}|=1$ in case of $\R$, and the $\lambda$-variations of $\Im G^{\lambda}_{\mathbf{T}_q}(x,y)$ also decay exponentially with $d(x,y)$. This is in contrast to $\left|\frac{\dd^j}{\dd k^j}\ee^{\ii km}\right| = m^j$ which grows with the distance. These differences make it more reasonable to consider the phase function as $\phi(k)= E(k)$ and keep the analog of $\ee^{\ii k(x-y)}$ in the observable part; the aforementioned control over its modulus and derivatives allows for a good control using the stationary phase method. This is the qualitative reason why we observe a fixed speed of dispersion $t^{-3/2}$ independently of the potentials $W$ and $\alpha$ we put on the edges/vertices.

\providecommand{\bysame}{\leavevmode\hbox to3em{\hrulefill}\thinspace}
\providecommand{\MR}{\relax\ifhmode\unskip\space\fi MR }
\providecommand{\MRhref}[2]{%
  \href{http://www.ams.org/mathscinet-getitem?mr=#1}{#2}
}
\providecommand{\href}[2]{#2}

\end{document}